\documentclass[preprint,12pt]{elsarticle}

\usepackage{amsmath,amsfonts,amsbsy,amstext,amsopn,amssymb}

\usepackage{longtable}
\usepackage{theorem}
\usepackage{algorithmic}
\usepackage{algorithm}
\usepackage{color}
\usepackage{subfig} 
\usepackage{graphicx}
\usepackage{lscape}

\makeatletter \@addtoreset{equation}{section}

\makeatother

{\theorembodyfont{\slshape}
\newtheorem{theorem}{Theorem}[section]

\newtheorem{lemma}{Lemma}[section]
\newtheorem{corollary}{Corollary}[section]

}
{\theorembodyfont{\rmfamily}
\newtheorem{definition}{Definition}

\newtheorem{remark}{Remark}
\newtheorem{example}{Example}

}

\newcommand{\vect}{{\sf{vec}}}

\newcommand{\rt}{{\top }}

\def\Oh{{\mathcal O}}

\def\rel{{\rm rel}}
\def\u{{\rm U}}

\def\bfH{{\mathcal M}}
\def\bfJ{{\mathcal J}}

\def\sfv{{v}}

\usepackage{multirow}

\newcommand {\eq} [1] {\begin{equation}\label{#1}}
\newcommand {\en} {\end{equation}}
\newcommand {\cM}       {{\cal M}}

\newcommand {\cX}       {{\cal W}} 
\newcommand {\cY}       {{\cal V}}

\newcommand {\R}        {{\mathbb R}}

\newcommand {\mat}      [1] {\left[\begin{array}{#1}}
\newcommand {\rix}          {\end{array}\right]}

\newcommand {\diag}     {\mathop{\sf Diag}\nolimits}

\newcommand {\cond}     {\mathop{\rm cond}\nolimits}
\newcommand {\trace}    {\mathop{\sf trace}\nolimits}

 \catcode`@=11
 \font\tenex=cmex10 
 \newdimen\p@renwd
 \setbox0=\hbox{\tenex B} \p@renwd=\wd0 
 \def\bmat#1{\begingroup \m@th
   \setbox\z@\vbox{\def\cr{\crcr\noalign{\kern2\p@\global\let\cr\endline}}%
     \ialign{$##$\hfil\kern2\p@\kern\p@renwd&\thinspace\hfil$##$\hfil
       &&\quad\hfil$##$\hfil\crcr
       \omit\strut\hfil\crcr\noalign{\kern-\baselineskip}%
       #1\crcr\omit\strut\cr}}%
   \setbox\tw@\vbox{\unvcopy\z@\global\setbox\@ne\lastbox}%
   \setbox\tw@\hbox{\unhbox\@ne\unskip\global\setbox\@ne\lastbox}%
   \setbox\tw@\hbox{$\kern\wd\@ne\kern-\p@renwd\left[\kern-\wd\@ne
     \global\setbox\@ne\vbox{\box\@ne\kern2\p@}%
     \vcenter{\kern-\ht\@ne\unvbox\z@\kern-\baselineskip}\,\right]$}%
   \null\;\vbox{\kern\ht\@ne\box\tw@}\endgroup}
 \catcode`@=12
\def\bu{{\bf a }}
\def\bv{{ \bf b }}
\def\bz{{z}}
\def\bw{{w }}

\def\bfH{{\mathcal N}}
\def\bfJ{{\mathcal H}}

\def\sd{{\sf     d}}

\newcommand{\call}{{{L }}}

\usepackage{lineno,hyperref}
\modulolinenumbers[5]

\journal{Journal of \LaTeX\ Templates}

\begin{document}

\begin{frontmatter}

\title{Mixed and componentwise condition numbers for a linear function of the solution of the total least squares problem}

\author[mymainaddress]{Huai-An Diao\corref{mycorrespondingauthor}
}
\cortext[mycorrespondingauthor]{Corresponding author}
\fntext[myfootnote]{Email address: hadiao@nenu.edu.cn, hadiao78@yahoo.com (H.A. Diao), suny235@nenu.edu.cn (Y. Sun). }

\author[mymainaddress]{Yang Sun}


\address[mymainaddress]{School of Mathematics and Statistics, Northeast
Normal University, No. 5268 Renmin Street, Chang Chun 130024, P.R.
of China.}

\begin{abstract}
In this paper,
we consider the mixed and componentwise condition numbers for a linear function of the solution to the total least squares (TLS) problem. We derive the explicit expressions of the mixed and componentwise condition numbers through the dual techniques. The sharp upper bounds for the derived  mixed and componentwise condition numbers are obtained. For the structured TLS problem, we consider the structured perturbation analysis and obtain the corresponding expressions of the mixed and componentwise condition numbers. We prove that the structured ones are smaller than their corresponding unstructured ones based on the derived expressions. Moreover, we point out that the new derived expressions can recover the previous results on the condition analysis for the TLS problem.  The numerical examples show that the derived condition numbers can give sharp perturbation bounds, on the other hand normwise condition numbers can severely overestimate the relative errors because normwise condition numbers ignore the data sparsity and scaling. Meanwhile, from the observations of numerical examples, it is more suitable to adopt structured condition numbers to measure the conditioning for the structured TLS problem.
\end{abstract}


\end{frontmatter}

\section{Introduction}
\setcounter{equation}{0}
 For a given  over-determined set of $m$ linear equations $A x \approx b$ in $x \in \mathbb{R}^n$, the total least squares (TLS)   problem \cite{GolubTLS1980,GolubVanLoan2013Book,VanHuffelVandewalle1991Book} is defined by
\begin{eqnarray}\label{TLS:definition}
{\rm minimize}&&\quad\left\|[A, ~b]-[\widehat{A}, ~\widehat{b}]\right\|_F\\
\mbox{subject to}&&\quad \widehat{b} \in {\sf R}(\widehat{A}),\, [\widehat{A}, ~\widehat{b}]\in \R^{m\times(n+1)},\nonumber
\end{eqnarray}
where $A\in \R^{m\times n}$, $b\in \R^{m}$, ${\sf R}(A)$ denotes the range space of the matrix $A$ and $\|\cdot\|_F$ is the Frobenius norm.  Let $[\widehat{A}, ~\widehat{b}]$ is  a minimizer of \eqref{TLS:definition}, then any $x$ satisfying $\widehat{A}x=\widehat{b}$ is called a TLS solution and $[\widehat {\Delta {A}}, ~\widehat{\Delta {b}}]=[A, ~b]-[\widehat{A}, ~\widehat{b}]$ is the corresponding TLS correction.  In order to guarantee the existence and  uniqueness of the TLS solution to \eqref{TLS:definition}, the {\em  genericity} condition \eqref{eq:genericity} of the TLS problem was first introduced in \cite{GolubTLS1980}.  In this paper, we always assume  that the genericity condition holds (for more information about the {\em nongeneric} problem, see, e.g.\cite{VanHuffelVandewalle1991Book}). The TLS problem has many applications in computer vision, image reconstruction, speech and audio processing, modal and spectral analysis,  linear system theory, and system identification, etc; see the review paper \cite{MarkovskyVanHuffel2007} of TLS for more details.

Now we first review numerical methods for TLS. When the size of TLS problem is small or medium, a classical direct solver based on the full singular value decomposition (SVD) of the augmented matrix $[A, ~b]$ can be adopted; see \cite[Algoritm 3.1]{VanHuffelVandewalle1991Book}. The solution in the generic case can be obtained from the right singular vector corresponding to the smallest singular value of $[A, ~b]$. The improved version of the SVD-based direct method can be implemented by using the partial SVD of $[A,~b]$ to compute in an efficient and reliable way a basis of the left and/or right singular subspace of a matrix associated with its smallest singular values; see \cite[Chapter 4]{VanHuffelVandewalle1991Book} for more details. The iterative method combining the Rayleigh quotient iteration and preconditioned conjugate gradient method was proposed in \cite{Bjorck2000SIMAX} for the large-scale and sparse TLS problem. A lot of researchers had paid attentions to the numerical solver for the large-scale structured TLS problem; see the papers \cite{BeckSIAM2005,BeckSIAM2007,MastronardiEtal2001,KammNagy1998,LemmerlingVanHuffel2002,LemmerlingVanHuffel2001}.
%

In sensitivity analysis, the condition number is considered as a fundamental tool since they describes the worst-case sensitivity of the solution to a problem with respect to small perturbations in the input data. The problem with a large condition number is called an {\em ill-posed} problem  (cf. \cite{Higham2002Book}). Since the 1980's, there had been  some papers related to  the perturbation analysis for the TLS problem; see \cite{FierroBunch,GolubTLS1980,Weib} and the references therein. As far as we know, the general normwise condition numbers were studied by  Rice \cite{Rice}, which measure errors for both input and output data by means of norm. However, when the data is badly-scaled or sparse, normwise condition numbers may allow large relative perturbations on small entries and give over-estimated perturbation bounds. To overcome the shortcoming of  normwise condition numbers, componentwise perturbation analysis has been extensively studied for many classical problems in matrix computation; see the comprehensive survey \cite{Higham1993CompSurvey} and the references therein. Because rounding errors for the data in the floating point system are measured  componentwisely, it is more reasonable to adopt componentwise perturbation analysis and more sharper bounds can be obtained through componentwise perturbation analysis. In fact, most error bounds in LAPACK
\cite{Anderson1999Lapack} are based on componentwise perturbation analysis. In componentwise perturbation analysis, two types of condition numbers, described as mixed and componentwise,  were proposed; see \cite{CuckerDiaoWei2007MC,GohbergKoltracht1993SIMAX,Rohn89,Skeel79} for details.

 Under the genericity condition, normwise condition numbers for the TLS problem had been studied in \cite{Baboulin2011SIMAX,JiaLi2013}. Specifically, the explicit expressions, their lower and upper bounds, replying on the SVDs of the matrix $A$ and/or the augmented matrix $[A,~b]$, were derived. However, as stated in the previous paragraph, when the data is sparse or badly scaled, normwise condition numbers may heavily over-estimate the conditioning of the TLS problem. Thus it is necessary to consider the conditioning of the TLS problem using componentwise perturbation analysis, which had been done in \cite{Zhou}. As shown \cite[Example 1]{DiaoWeiXie}, there are big differences between normwise condition numbers and mixed/componentwise condition numbers, which again confirms that it is  necessary to  study mixed and componentwise condition numbers for the TLS problem. Moreover, when the TLS problem is structured, it is suitable to study the structured perturbation analysis because this will help us to understand the structured preserved algorithms; see \cite{LiJia2011}. Structured perturbation analysis for linear system, linear least squares and Tikhonov regularization problem had been investigated in \cite{Higham1998structured,Rump03a,Rump03b,CuckerDiao2007Calcolo,DiaoWeiQiao2016,XuWeiQiao}, respectively.

In this paper,  under the genericity condition, we study the sensitivity  of a linear function of the TLS solution $x$ to perturbations on the date $A$ and $b$, which is defined as
\begin{align}\label{eq:g dfn}
\Psi: \R^{m\times n} \times
\R^m   &\rightarrow \R^k\\
\Psi(A,\,b)&:=\call x,\nonumber
\end{align}
where $x$ is the unique solution to the TLS problem \ref{TLS:definition}, and $\call$ is an $k$-by-$n$, $k \le n$, matrix introduced for
the selection of the solution components. For example, when
$\call= I_n$ ($k=n$), all the $n$ components of the solution $x$
are equally selected. When $\call= e_i$ ($k=1$), the $i$th unit
vector in $\R^n$, then only the $i$th component of the
solution is selected. In the reminder of this paper, we always suppose that $\call$ is not numerically perturbed. Condition numbers for a linear function of the solution to linear system \cite{ChandrasekaranIpsen1993}, linear least squares \cite{ArioliBaboulinGratton2007,BaboulinGratton2009BIT,BaboulinDongarraGratton2009} and the TLS problem \cite{Baboulin2011SIMAX} had been studied extensivley. Contrary to \cite{Baboulin2011SIMAX} for the normwise condition number of the TLS problem, in this paper, we will consider mixed and componentwise condition numbers of a linear function of the TLS solution $x$ under unstructured and structured perturbations. For the structured TLS problem  \cite{BeckSIAM2005,BeckSIAM2007,MastronardiEtal2001,KammNagy1998,LemmerlingVanHuffel2002,LemmerlingVanHuffel2001}, we consider the case that $A$ is  a linear structure matrix, for example Toeplitz matrix.  Because the set $\cal S$ of the linear structured matrix is a subspace of $\R^{m\times n}$, we assume its dimension is $q$ and there exits a uniques vector denoted by $a$ such that
\begin{equation}\label{eq:A linear}
	A=\sum_{i=1}^q a_i S_i,
\end{equation}
where $S_1,\ldots, S_q$ is the basis of $\cal S$. In this paper, we also study the sensitivity  of a linear function of the structured TLS solution $x$ to perturbations on the date $a$ and $b$, which is defined as
\begin{align}\label{eq:g dfn s}
\Psi_s: \R^{q} \times
\R^m   &\rightarrow \R^k\\
\Psi(a,\,b)&:=\call x, \nonumber
\end{align}
where $x$ is the unique solution to the structured TLS problem under genericity condition.

This paper is devoted to obtain the explicit expressions for   mixed and componentwise condition numbers of the linear function of the TLS solution when perturbations on data are measured componentwise and the perturbations on the solution are measured either componentwise or normwise by means of the dual techniques \cite{BaboulinGratton2009BIT}. In particular, as also mentioned in \cite{BaboulinGratton2009BIT}, the dual techniques enable us to derive condition numbers by maximizing a linear function over a space of smaller dimension than the data space. Both the unstructured and structured  condition numbers  are considered. We also study the relationship between those two type condition numbers, and prove that the structured ones are smaller than the unstructured ones from the derived expressions. Moreover,  the expressions of the proposed condition numbers can recover the pervious results \cite{Zhou,LiJia2011} when $L=I_n$. By taking account of the SVD method for solving TLS, we give SVD-based formulae of the proposed condition numbers. Numerical examples show that our theoretical results are effective. Especially,  in Example \ref{ex:1} our unstructured mixed and componentwise condition numbers for a linear function of the TLS solution $x$ can be much smaller than the nowise condition number given by \cite{Baboulin2011SIMAX}, which means that  it is more suitable to use the mixed and componentwise condition numbers to measure the conditioning of TLS when the data is sparse or badly-scaled. For the structured TLS problem, Example \ref{ex:str} shows it is necessary to adopt the structured mixed and componentwise condition numbers instead of using the unstructured ones to measure the conditioning for the structured TLS problem.

The paper is organized as follows. In Section \ref{sec:pre},  some basic result of the TLS problem and the dual techniques for deriving condition number  \cite{BaboulinGratton2009BIT} are reviewed. We derive the explicit expression for the unstructured and structured condition numbers, and study the relationship between them. Also we prove that our results can recover the previous condition numbers expressions of the TLS problem when $L=I_n$. Sharp upper bounds for unstructured mixed and componentwise condition numbers are also given. Moreover, by taking account of the SVD method for solving the TLS problem, we obtain SVD-based fomula for our proposed condition numbers.  We do some numerical examples to show the effectiveness of the proposed condition numbers in Section \ref{sec:nume ex}.  At end, in Section \ref{sec:con} concluding remarks are drawn.



\section{Preliminaries}\label{sec:pre}
In this section we  give some backgrounds on theoretical results on the TLS problem. Also, the dual techniques for deriving condition number's expressions are reviewed.

\subsection{Basic results}
Assume that we have the SVDs of $A$ and $[A, ~b]$, respectively,
\begin{equation}\label{svd}
A = \widetilde{U} \widetilde{\Sigma} \widetilde{V}^\top,\quad [A, ~b]=U\Sigma V^ \top,
\end{equation}
where $U, ~\widetilde{U} \in \R^{m\times m},\, \widetilde{V} \in \R^{n\times n}$ and $V\in \R^{(n+1)\times (n+1)}$ are orthogonal, $\widetilde{\Sigma}=\diag(\widetilde{\sigma}_1,\widetilde{\sigma}_2,
\ldots,\widetilde{\sigma}_n)\in \R^{m\times n},\, \widetilde{\sigma}_1 \geq \widetilde{\sigma}_2 \geq \ldots \geq \widetilde{\sigma}_n \geq 0$ and $\Sigma=\diag(\sigma_1,\ldots,\sigma_n,\sigma_{n+1})\in \R^{m\times (n+1)},\, \sigma_1\geq\sigma_2\geq\ldots\geq\sigma_{n+1}\geq 0.$
Let $v_{n+1}$ be the last column of $V$ and  $v_{n+1,n+1}$ denotes its $(n+1)$th component of $v_{n+1}$.
It is assumed that the {\em genericity} condition
\begin{equation}\label{eq:genericity}
\widetilde{\sigma}_n>\sigma_{n+1},
\end{equation}
holds to ensure the existence and uniqueness of the TLS solution \cite{GolubTLS1980}.
As mentioned in \cite[Page 35]{VanHuffelVandewalle1991Book}, the genericity condition \eqref{eq:genericity} is equivalent to
$\sigma_n >\sigma_{n+1} \mbox{ and } v_{n+1,n+1}\neq 0$.
The following identities hold for the TLS solution $x$ (cf. \cite{GolubTLS1980})
\begin{align}\label{ex:x}
 \begin{bmatrix}x\cr -1\end{bmatrix} = -\frac{1}{v_{n+1,n+1}}v_{n+1},\quad   v_{n+1,n+1}=\frac{1}{\sqrt{1+x^\top x}}.
\end{align}
It follows from \cite[Page 36, Theorem 2.7]{VanHuffelVandewalle1991Book} that the TLS solution $x$ satisfies the equation
\begin{equation}\label{eq:norml eq}
\left( A^\top A-\sigma_{n+1}^2 I_n \right) x=A^\top b.
\end{equation}
From the SVD of $[A,~b]$, it is easy to check that
\begin{equation}\label{eq:r}
  r=b-Ax=-[A,~b]\begin{bmatrix}
    x\cr -1
  \end{bmatrix}=\frac{1}{v_{n+1,n+1}}[A,~b]v_{n+1}=\frac{\sigma_{n+1}}{v_{n+1,n+1}}u_{n+1},
\end{equation}
where $u_{n+1}$ is the $(n+1)$th column of $U$.
%


Lemma \ref{lemma:P} presents an explicit expression for the inverse of $P$ in  \eqref{eq:norml eq}.

\begin{lemma}\label{lemma:P}\cite[Lemma 2]{DiaoWeiXie}
Let $P=A^{\top}A-\sigma_{n+1}^2 I_{n}$. Under the genericity condition \eqref{eq:genericity}, recalling $x$ given by \eqref{ex:x}, it holds that
\begin{align*}
P^{-1}&=Q_1 Q Q_1, \nonumber
\end{align*}
where $Q=V_{11} D^{-1 }V_{11}^\top$,  $V_{11}$ is the leading $n\times n$ submatrix $V$ in \eqref{svd},  $Q_1=I_n+xx^\top$, $D=\diag(\sigma_1^2-\sigma_{n+1}^2,\sigma_2^2-\sigma_{n+1}^2,\ldots, \sigma_n^2-\sigma_{n+1}^2)$.
\end{lemma}

The classical direct solver of the  TLS solution $x$ of \eqref{TLS:definition} is to calculate the SVD of $[A,~b]$. The detailed description of this computation is shown in \cite[Algorithm 3.1]{VanHuffelVandewalle1991Book}. When the full SVD of $[A,~b]$ is computed,  the TLS solution $x$ can be computed from \eqref{ex:x} and $P^{-1}$ can be computed efficiently via Lemma \ref{lemma:P}, which help us to derive SVD-based expressions of the proposed condition numbers.

The Fr\'{e}chet derivatives of the function $\Psi$ with
respect to the input data $[A,~b]$ plays an important role in deriving condition numbers expressions, which is given in the following lemma. Let $\sd \Psi([A,~b] )$ be the Fr{\'e}chet derivative of $\Psi$ at $[A,~b] $.

\begin{lemma} \cite[Proposition 1]{Baboulin2011SIMAX}\label{lemma:gLSQI}
Under the genericity condition \eqref{eq:genericity}, the function $\Psi$ is a continuous mapping on $\R^{m\times n} \times \R^{m\times 1} $. In addition, $\Psi$ is Fr\'{e}chet
differentiable at $(A,\, b)$ and its Fr\'{e}chet derivative is given by
\begin{align}\label{eqnJ}
J&:=\sd {\Psi}(A,\,b)\cdot (\sd A,\, \sd b)=\call P^{-1} \left[(\sd A)^\top r- \left[A^\top +\frac{2xr^\top}{1+x^\top x}\right]\sd A x\right]\nonumber\\
&\quad + \call P^{-1} \left[A^\top+\frac{2xr^\top}{1+x^\top x}\right] \sd b\nonumber \\
&:=J_{1}(\sd A)+J_{2}(\sd b),
\end{align}
where $\sd A \in \R^{m\times n},\, \sd b \in \R^{m\times 1}$.
\end{lemma}

Given the perturbations $\Delta A$ of $A$ and $\Delta b$ of $b$. Under the genericity condition \eqref{eq:genericity}, when $\|[\Delta A,~\Delta b]\|_F $ is small enough, the perturbed TLS problem
\begin{eqnarray}\label{eq:pert TLS}
{\rm minimize}&&\quad\left\|[A+\Delta A, ~b+\Delta b]-[\widehat{A}, ~\widehat{b}]\right\|_F\\
\mbox{subject to}&&\quad \widehat{b} \in {\sf R}(\widehat{A}),\, [\widehat{A}, ~\widehat{b}]\in \R^{m\times(n+1)},\nonumber
\end{eqnarray}
has  a unique TLS solution $x+\Delta x$. The {\em absolute} normwise condition number \cite{Baboulin2011SIMAX} of $\Psi$ is defined by
$$
\cond(L, A,b)=\lim_{\epsilon \rightarrow
0}\sup_{\substack{\left\|[\Delta A,~ \Delta b]\right\|_F \le\epsilon  }}\|L \Delta x\|_2  = \max_{[\Delta A,~\Delta b] \neq 0}\frac{\|L \sd {\Psi}(A,\,b)\cdot (\sd A,\, \sd b)\|_2}{\|[\Delta A,~\Delta b]\|_F},
$$
where $x+\Delta x$ is the TLS solution of \eqref{eq:pert TLS}, $\|A\|_2$ is the spectral norm of $A$ and the last equality is from \cite{Rice}.   Baboulin and  Gratton \cite{Baboulin2011SIMAX} derived the exact SVD-based expression of $\kappa(L, A,b)$ as follows
\begin{align}\label{pro:cond_Baboulin}
\cond(L, A,b)&=\sqrt{1+\|x\|_2^2}\left\|L \widetilde{V} D'[\widetilde{V}^\top ~{0}]V[D'' ~{ 0}]^\top\right\|_2,
\end{align}
where
\begin{align*}
  D'&=\diag\left((\widetilde{\sigma}_1^2-\sigma_{n+1}^2)^{-1},\ldots,(\widetilde{\sigma}_n^2-\sigma_{n+1}^2)^{-1}\right),\nonumber\\
D''&=\diag\left((\sigma_1^2+\sigma_{n+1}^2)^{\frac{1}{2}},\ldots,(\sigma_n^2+\sigma_{n+1}^2)^{\frac{1}{2}}\right).\nonumber
\end{align*}
The relative normwise condition number corresponding to $\cond(L, A,b) $ in \eqref{pro:cond_Baboulin} can be defined by
\begin{equation}\label{eq:relative norm}
\cond^{\rel}(L,A,b)= \lim_{\epsilon \rightarrow
0}\sup_{\substack{\left\|[\Delta A,~ \Delta b]\right\|_F \le\epsilon \left\|[A,~ b]\right\|_F    }}\frac{\|L \Delta x\|_2}{\|L x\|_2}=\frac{\cond(L,A,b)\|[A,~b]\|_F}{\| L x\|_2}.
\end{equation}

In the following, if $A\in \R^{m\times n}$ and $B \in \R^{p \times q}$, then the
{\em Kronecker product} $A \otimes B\in\R^{mp\times nq}$ is
defined by $A \otimes B = \left[a_{ij}B\right] \in \R^{mp\times nq}$ \cite{Graham1981book}. Zhou et al.~\cite{Zhou} defined and derived the relative mixed and componentwise condition numbers  as follows,
\begin{align}\label{eq:mixed}
  m(A,b)&=\lim_{\epsilon \rightarrow
0}\sup_{\substack{|\Delta A| \le\epsilon  |A|,\atop |\Delta b|\leq \epsilon|b| }}\frac{\|\Delta x\|_\infty}{\|x\|_\infty}=\frac{\left\|\left|M+N\right|\begin{bmatrix}
    |{\sf vec}(A)|\cr |b|
  \end{bmatrix}\right\|_\infty}{\|x\|_\infty},\\
  c(A,b)&=\lim_{\epsilon \rightarrow
0}\sup_{\substack{|\Delta A| \le\epsilon  |A|,\atop |\Delta b|\leq \epsilon |b| }}\left\|\frac{\Delta x}{x}\right\|_\infty =\left\|D_x^\dagger\left|M+N\right|
  \begin{bmatrix}
    |{\sf vec}(A)|\cr |b|
  \end{bmatrix}\right\|_\infty,\nonumber
\end{align}
where we denote by $|A|=(|a_{ij}|)$ for a given matrix $A$, $|a| \leq |b|$ represents $|a_i| \leq |b_i|$ for two vectors $a=[a_1,a_2, \ldots, a_n]^\top$ and $b=[b_1,b_2, \ldots, b_n]^\top$,
\begin{align*}
M & = \begin{bmatrix}P^{-1}\otimes b^\top-x^\top \otimes(P^{-1}A^\top )-P^{-1}\otimes(Ax)^\top &P^{-1}A^\top \end{bmatrix},\\
N & = 2\sigma_{n+1}P^{-1}x(v_{n+1}^\top \otimes u_{n+1}^\top),
\end{align*}
 the notation ${\sf vec}(A)$ stacks columns of $A$ one by one to a column vector, $D_x^\dagger$ is the Moore-Penrose inverse \cite{GolubVanLoan2013Book} of the diagonal matrix $D_x$  with $(D_x)_{ii}=x_i$, $\|\cdot\|_\infty$ is the infinity norm and the symbol $\frac{y}{x}$ denotes the componentwise division of $x,\,y$, assuming that if $x_i=0$ for some index $i$ then $y_i$ should be zero.

The structured condition numbers for the TLS problem with linear structures were studied by
Li and Jia in \cite{LiJia2011} We first review the structured perturbation results given in \cite{LiJia2011}. Recall when $A\in \cal S$, $A$ can be determined by \eqref{eq:A linear}.
Denote
\begin{align}
\cM^{st}&=\begin{bmatrix}
	\vect(S_1) &\cdots  & \vect(S_q)
\end{bmatrix},\quad \cM^{st}_{A, b} = \begin{bmatrix}
  \cM^{st} & 0\cr
 0  &I_m
\end{bmatrix}, \notag \\
K &=P^{-1}
\left(2 A^\top \frac{rr^\top }{\|r\|_2^2} G(x)-A^\top G(x) + \begin{bmatrix}
	I_n \otimes r^\top  & 0
\end{bmatrix} \right ),\quad  G(x)=\begin{bmatrix}
	x^\top &-1
\end{bmatrix} \otimes I_m.  \label{eq:K}
\end{align}
  The structured mixed condition number $m_s(A,b)$ is characterized
as
\begin{eqnarray}\label{eq:str mixed}
m_s(A,b) &=& \lim_{\epsilon \rightarrow
0}\sup_{\substack{ |\Delta A| \le\epsilon  |A|,\, |\Delta b| \le\epsilon  |b|  \\ \Delta A\in {\cal S}}}\frac{\|\Delta x\|_\infty }{\|x\|_\infty } = \frac{\left\|\left |K \cM^{st}_{A, b} \right| \begin{bmatrix}
	|a| \\ |b|
\end{bmatrix} \right\|_\infty  }{\|x\|_\infty },
\end{eqnarray}
and they also proved that $m_s(A,b) \leq  m(A,b)$.

\subsection{Dual techniques}


Let  $\cX$  and  $\cY$  be the Euclidean spaces  equipped scalar products $\langle \cdot ,\cdot \rangle_{\cX}$ and  $\langle\cdot,\cdot \rangle_{\cY}$ respectively, and we consider  a linear operator $ {\cal L}:\, \cX \rightarrow \cY $. We denote $\|\cdot\|_{\cX} $ and $\|\cdot\|_{\cY}$ by the corresponding norms of $\cX$  and  $\cY$,  respectively. The well-known adjoint operator and dual norm are defined as follows.

\begin{definition}\label{def:adj dual}
The adjoint operator of  $\cal L$,\, ${\cal L}^{\ast}:\cY \rightarrow \cX $  is defined by
\[
\langle \bv, {\cal L}\bu \rangle_{\cY}=\langle {\cal L}^{\ast}\bv,\bu\rangle_{\cX}
\]
\end{definition}
where $\bu\in \cX $ and $\bv \in\cY $.  The dual norm $ \|\cdot\|_{\cX^{\ast}}$  of  $ \|\cdot\|_{\cX} $  is defined by
\[
\|\bu\|_{\cX^{\ast}}=\max_{\bw\neq0}\frac{ \langle  \bu,\bw \rangle _{\cX}}{\|\bw\|_{\cX}}
\]
and the dual norm $\|.\|_{\cY^{\ast}}$ can be defined similarly.

Using the canonical scalar product in $ \R^{n} $, the corresponding dual norms  with respect to the common vector norms are  given by :
\[
\|\cdot \|_{{1}^\ast}=\|\cdot\|_{\infty}, \quad  \|\cdot \|_{{\infty}^\ast}=\|\cdot\|_{1}\quad \mbox{and } \quad \|\cdot\|_{{2}^\ast}=\|\cdot\|_{2}.
\]
Let the scalar product $ \langle A,B\rangle = \trace(A^\top B)$ be defined in  $\R^{m\times n}$, where $\trace(A)$ is the trace of $A$. Then it is easy to see that   $ \|A\|_{F}\ast=\|A\|_{F}$ since  $\trace(A^\top A)=\|A\|_{F}^{2}$.

For the linear operator $\cal L$ from $\cX$  to $\cY$, let $\|\cal L\|_{\cX,\cY}$  be the operator norm induced by
the norms  $\|\cdot\|_{\cX}$ and  $\|\cdot\|_{\cY}$. Consequently, for linear operators from $\cY$ to $\cX$, the norm
induced from the dual norms $\| \cdot \|_{\cX^\ast}$ and
$\| \cdot \|_{\cY^\ast}$, is denoted by $\| \cdot \|_{\cY^\ast,\cX^\ast}$.

We have the following  result for the adjoint operators and dual norms \cite{BaboulinGratton2009BIT}.
\begin{lemma}\label{lemma:adjoint}
With notations above, the following property $$
\|\cal L\|_{\cX,\cY}=\|{\cal L}^{\ast}\|_{\cY^\ast,\cX^\ast}
$$
holds.
\end{lemma}

As pointed in \cite{BaboulinGratton2009BIT}, it may be desirable  to compute $ \|{\cal L}^{\ast}\|_{\cY^\ast,\cX^\ast} $ instead of $\|\cal L\|_{\cX,\cY}$
when the dimension of the Euclidean space $\cY^{\ast}$  is  lower  than  $\cX$ because it implies a maximization over a space of smaller dimension.

Now, we consider a product space  $ \cX=\cX_{1}\times \cdots\times \cX_{s} $  where each Euclidean space  $\cX_{i}$ is equipped with the scalar product  $\langle\cdot,\cdot \rangle_{\cX_{i}}$ and the corresponding norm $\|\cdot\|_{\cX_{i}}$. In  $\cX$, we define the following scalar product
\[
\langle (\bu_{1},\cdots,\bu_{s}),(\bv_{1},\cdots,\bv_{s})\rangle=\langle \bu_{1},\bv_{1}\rangle_{\cX_{1}}+\cdots+\langle \bu_{s},\bv_{s}\rangle_{\cX_{s}},
\]
and the corresponding product norm
\[
\|(\bu_{1},\cdots,\bu_{s})\|_{v}=v(\|\bu_{1}\|_{\cX_{1}},\cdots,\|\bu_{s}\|_{\cX_{s}}),
\]
where   $v$  is an absolute norm on  $\R^{s}$, that is  $ v(|\bu|)=v(\bu)$, for any  $\bu\in \R^{s}$, see \cite{Higham2002Book} for details. We denote   $v^{\ast}$ is the dual norm of $v$
with respect to the canonical inner-product of $ \R^{s} $  and we are interested in determining the dual $ \|\cdot\|_{v^\ast}$ of the
product norm $\|\cdot\|_{v}$  with respect to the scalar product of  $\cX$. The following result can be found in \cite{BaboulinGratton2009BIT}.

\begin{lemma}\label{lemmaProductNorm}
The dual of the product norm can be expressed by
\[
\|(\bu_{1},\cdots,\bu_{s})\|_{v^\ast}=v(\|\bu_{1}\|_{\cX_{1^\ast}},\cdots,\|\bu_{s}\|_{\cX_{s^\ast}}).
\]
\end{lemma}

In the following we apply adjoint operators
and dual norms to derive the explicit expressions for the condition numbers of TLS. We can view the Euclidean space $\cX$ with
norm $\| \cdot \|_\cX$ as the space of the input data in TLS and $\cY$ with norm $\| \cdot \|_\cY$ as the space
of the solution in TLS. Then
the function $\Psi$ in (\ref{eq:g dfn}) is an operator from $\cX$
to $\cY$ and the condition number is the measurement of the
sensitivity of $\Psi$ to the perturbation in its input data.

From \cite{Rice}, if $\Psi$ is Fr{\'e}chet differentiable in neighborhood of  $\bu \in \cX$, then the absolute condition number of  $\Psi$ at $\bu\in \cX$  is given by
\[
\kappa(\bu)=\|\sd \Psi(\bu)\|_{\cX,\cY}=
\max_{\| \bz \|_\cX = 1} \|\sd \Psi( \bu ) \cdot \bz \|_\cY ,
\]
where $ \|\cdot\|_{\cX,\cY} $  is the operator norm induced by the norms  $ \|\cdot\|_{\cX} $ and  $ \|\cdot\|_{\cY}$ and $\sd \Psi(\bu )$ is the Fr{\'e}chet derivative of $\Psi$ at $\bu$. If $ \Psi(\bu) $ is nonzero, the \emph{relative condition number} of  $\bu$ at  $\bu\in \cX$ is defined as
\[
\kappa^{\rel}(\bu)=\kappa(\bu)\frac{\|\bu\|_{\cX}}{\|\Psi(\bu)\|_{\cY}}.
\]
The expression of  $\kappa(\bu)$  is related to the operator norm of the linear operator  $\sd \Psi(\bu)$. Applying
Lemma~\ref{lemma:adjoint}, we have the following expression of
$\kappa(\bu)$ in terms of adjoint operator and dual norm:
\begin{equation} \label{eqnK}
 \kappa(\bu)=\max_{\|\sd \bu\|_{\cX}=1} \|\sd \Psi(\bu)\cdot \sd \bu\|_{\cY}=\max_{\|\bz \|_{\cY^\ast}=1}\|\sd \Psi(\bu)^{\ast}\cdot \bz\|_{\cX^\ast}.
\end{equation}

Now we consider the componentwise metric on a data space $ \cX=\R^{n}$. For any given $ \bu\in \cX$,  the subset $ \cX_{\bu} \in \cX$ is a set of all elements $ \sd \bu\in \cX$ satisfying that $\sd \bu_{i}=0$  whenever  $\bu_{i}=0$, $1\leq i\leq n$. Thus in a componentwise perturbation analysis, we measure the perturbation  $\sd \bu\in \cX_{\bu}$ of  $\bu$  using the following componentwise norm with respect to $\bu$
\begin{equation}\label{eq:rr}
 \|\sd \bu\|_{c}=\min\{\omega,|\sd \bu_{i}|\leq \omega |\bu_{i}|,i=1,\ldots, n\}.
\end{equation}
Equivalently, it is easy to see that  the componentwise relative norm has the following property
\begin{equation}\label{eqnCNorm}
\|\sd \bu\|_{c}=\max\left\{\frac{|\sd \bu_{i}|}{|\bu_{i}|},\bu_{i}\neq 0\right\}=\left\|\left(\frac{|\sd \bu_{i}|}{|\bu_{i}|}\right)\right\|_{\infty},
\end{equation}
where $\sd\bu \in \cX_\bu$.

In the following we consider the dual norm $\| \cdot \|_{c^\ast}$
of the componentwise norm $\| \cdot \|_c$.
Let the product space $\cX$
be $\mathbb{R}^n$, each $\cX_i$ be $\mathbb{R}$,
and the absolute norm $v$ be $\| \cdot \|_{\infty}$.
Setting the norm $\| \sd \bu_i \|_{\cX_i}$ in $\cX_i$ to
$| \sd \bu_i | / | \bu_i |$ when $\bu_i \neq 0$, from
Definition~\ref{def:adj dual}, we have the dual norm
\[
\| \sd \bu_i \|_{\cX_i^\ast} =
\max_{\bz \neq 0} \frac{| \sd \bu_i \cdot \bz |}{\| \bz \|_{\cX_i}} =
\max_{\bz \neq 0} \frac{| \sd \bu_i \cdot \bz |}{|\bz| / |\bu_i|} =
|\sd \bu_i| \, |\bu_i| .
\]
Applying Lemma~\ref{lemmaProductNorm} and (\ref{eqnCNorm}) and recalling
$\| \cdot \|_{\infty^*} = \| \cdot \|_1$, we derive the explicit expression of the dual norm
\begin{equation} \label{eqnDualC}
\| \sd \bu \|_{c^\ast} =
\| (\|\sd \bu_1\|_{\cX^\ast} ,..., \|\sd \bu_n\|_{\cX^\ast}) \|_{\infty^\ast} =
\| (|\sd \bu_1| \, |\bu_1| ,..., |\sd \bu_n| \, |\bu_n|) \|_1 .
\end{equation}

Because of the condition $\| \sd \bu \|_\cX = 1$ in the
condition number $\kappa(\bu)$ in (\ref{eqnK}), whether
$\sd \bu$ is in $\cX_\bu$ or not, the expression
of the condition number $\kappa(\bu)$ remains valid. Indeed,
if $\sd \bu \not\in \cX_\bu$, that is, $\sd \bu_i\neq 0$
 while $\bu_i=0$ for some $i$, then $\| \sd \bu \|_c = \infty$.
Consequently, such perturbation $\sd \bu$ is
excluded from the calculation of $\kappa(\bu)$. Following
(\ref{eqnK}), we have the following lemma on the condition
number in adjoint operator and dual norm.

\begin{lemma} \label{lemmaK}
Using the above notations and the componentwise norm defined
in (\ref{eqnCNorm}), the condition number $\kappa(\bu)$ can be
expressed by
\[
\kappa(\bu) =
\max_{\| \bz \|_{\cY^\ast} = 1}
\| (\sd \Psi(\bu))^* \cdot \bz \|_{c^\ast} ,
\]
where $\| \cdot \|_{c^\ast}$ is given by (\ref{eqnDualC}).
\end{lemma}

In the next section, based on Lemma \ref{lemmaK}, the explicit expressions for condition numbers can be deduced, where we measure the errors for the solution using componentwise perturbation analysis, while for the input data, we can measure the error either
componentwise or normwise.  However, regardless of the norms chosen
in the solution space, we always use the componentwise norm in
the data space.

\section{Mixed and componentwise condition numbers for TLS}\label{sec:cond}

In this section we will derive the explicit condition numbers expressions for a linear function of the solution of TLS by means of the dual techniques under componentwise perturbations, which is introduced in \cite{BaboulinGratton2009BIT}. Both the unstructured and structured condition number expressions are derived. Moreover,  our condition numbers can recover the previous results on the mixed and componentwise condition numbers \cite{LiJia2011,Zhou} when we take $L=I_n$. Also sharp upper bounds for the unstructured mixed and componentwise condition numbers are obtained. Through using the already computed SVD for solving TLS, we can obtain SVD-based formulae for condition numbers and their upper bounds.

\subsection{Unstructured condition number expressions of TLS via dual techniques}

In this subsection we will derive the explicit expressions of unstructured condition numbers for TLS through dual techniques stated in the previous section.  Also we prove that the derived expressions and the previous ones \cite{Zhou} are mathematically equivalent. Sharp upper bounds absence of Kronecker product for condition numbers are given. Before that, we need the following lemma.

Using the definition of the adjoint operator and
the classical definition of the scalar
product in the data space $\R^{m\times n} \times \R^{m\times 1} $, an explicit
expression of the adjoint operator of the above $J(\sd A,\, \sd b)$ is given in the following lemma.

\begin{lemma}\label{lemmaDualJ}
The adjoint of operator of the Fr{\'e}chet derivative $J(\sd A,\, \sd b)$ in \eqref{eqnJ} is given by
\begin{align*}
J^{\ast}&:\R^{k}\rightarrow \R^{m\times n} \times \R^{m\times 1} \\
 &u\mapsto \left( r u^\top\call P^{-1} - \left[A^\top+\frac{2xr^\top}{1+x^\top x}\right] ^\top P^{-1}\call^\rt u x^\top, \, \left[A^\top +\frac{2xr^\top}{1+x^\top x}\right]^\top P^{-1} \call^\top u\right).
\end{align*}
\end{lemma}

\begin{proof}
Using \eqref{eqnJ} and the definition of the scalar product in the matrix space, for any $u\in \R^{k}$, we have
\begin{align*}
\langle u,J_{1}(u)\rangle=&u^\top \call P^{-1}\left[(\sd A)^\top r- \left[A^\top+\frac{2xr^\top}{1+x^\top x}\right]\sd A x \right]\\
=&\trace\left(ru^\top\call P^{-1}
(\sd A)^\rt \right)-\trace \left(xu^\top \call\ P^{-1}\left[A^\top +\frac{2xr^\top}{1+x^\top x}\right]\sd A\right)\\
=&\left\langle r u^\top\call P^{-1} -\left [A^\top+\frac{2xr^\top}{1+x^\top x}\right]^\top P^{-1}\call^\rt u x^\top, \sd A\right\rangle.
\end{align*}
For the second part of the adjoint of the derivative  $J$, we have
\begin{align*}
\langle u,J_{2}(u)\rangle=&u^\top \call P^{-1}\left[A^\top+\frac{2xr^\top}{1+x^\top x}\right]\sd d\\
=&\left\langle \left[A^\top +\frac{2xr^\top}{1+x^\top x}\right] ^\top P^{-1} \call^\top u , \sd b\right\rangle.
\end{align*}
Let
\begin{align*}
J_1^*({u})& =r u^\top\call P^{-1} - \left[A^\top+\frac{2xr^\top}{1+x^\top x}\right]^\top  P^{-1} \call^\rt u x^\top,\quad J_2^*({u}) =\left [A^\top+\frac{2xr^\top}{1+x^\top x}\right]^\top P^{-1} \call^\rt u \\
\end{align*}
then $$
\langle J^*({u}),\ (\sd A,\, \sd b) \rangle =
\langle (J_1^*({u}),\, J_2^*({u}))) , \
(\sd A,\, \sd b) \rangle = \langle {u},\ J(\sd A,\, \sd b)  \rangle,
$$
which completes the proof. \hfill $\Box$
\end{proof}

In fact, Lemma \ref{lemmaDualJ} establishes the same expressions of the adjoint operator of $J$ as that in Proposition 3 of \cite{Baboulin2011SIMAX}. However, we use a different proof here to avoid forming the explicit Kronecker product-based matrix expression of $J$, which appeared in Proposition 2 of \cite{Baboulin2011SIMAX}.

After obtaining an explicit expression of the adjoint operator
of the Fr{\'e}chet derivative,  we now give an explicit expression
of the condition number $\kappa$ (\ref{eqnK}) in terms of
the dual norm in the solution space in the following theorem.


\begin{theorem} \label{thmK}
The condition number for the TLS
problem can be expressed by
\[
\kappa = \max_{\| {u} \|_{\cY^\ast} = 1}
\left\| [ \bfH D_A \ \ \bfJ D_{{b}}\ \ ]^{\top}
\call^\top\right\|_{\cY^\ast,1} ,
\]
where
\begin{align} \label{eqnV}
 \bfH= P^{-1}\left[-x^\top \otimes \left(A^\top+\frac{2xr^\top}{1+x^\top x} \right )+I_n\otimes r^\top \right],\quad
 \bfJ=P^{-1}\left[A^\top+\frac{2xr^\top}{1+x^\top x}\right].
\end{align}
\end{theorem}

\begin{proof}
Let $\sd a_{ij}$, $\sd b_{ij}$,  be the entries of $\sd A, \,\sd b $ and  $\sd d$ respectively, using \eqref{eqnDualC}, we have
 \[
 \|(\sd A,\sd b)\|_{c^\ast}=\sum_{i,j}|\sd a_{ij}||a_{ij}|+\sum_{i,j}|\sd b{ij}||b_{ij}|.
 \]
Applying Lemma~\ref{lemmaDualJ}, we derive that
\begin{align*}
\|J^{\ast}(u)\|_{c\ast}=&\sum_{j=1}^{n}\sum_{i=1}^{m}  |a_{ij}|\left|\left(r u^\top\call P^\rt - \left( A^\top+ \frac{2xr^\top}{1+x^\top x} \right)^\top P^{-1} \call^\top u x^\top\right)_{ij}\right|                 \\+&\sum_{i=1}^{m}
|b_{i}|\left|\left(\left( A^\top+\frac{2xr^\top}{1+x^\top x}\right) P^{-1}\call^\top u\right)_{i}\right|\\
=&\sum_{j=1}^{n}\sum_{i=1}^{m}  |a_{ij}| \left|\left[r_i P^{-1} e_j - x_j P^{-1} \left( A^\top+\frac{2xr^\top}{1+x^\top x}\right) e_i\right]^\top \call^\top u \right|\\+&\sum_{i=1}^{m}
|b_{i}|\left|\left(P^{-1} (A^\top+\frac{2x r^\top}{1+x^\top x}) e_i\right) \call^\top u\right|
\end{align*}
where $r_{i}$ is the $i$th component of $r$. Then it can be verified that
$$
r_i P^{-1} e_j - x_j P^{-1} \left( A^\top+\frac{2xr^\top}{1+x^\top x}\right) e_i
$$
is the $(m(j-1)+i)$th column of the $n \times (mn)$ matrix
$\bfH$
 implying that the above expression
equals to
\[
\left\| \left[ \begin{array}{c}
D_A \bfH^\top \call^\top  {u} \\
D_{{b}} \bfJ^\top \call^\top  {u}\\
\end{array} \right] \right\|_1 =
\left\| [ \bfH D_A \ \ \bfJ D_{{b}}]^{\top}
\call^\top {u} \right\|_1 .
\]
The theorem then follows from Lemma~\ref{lemmaK}. \hfill $\Box$
\end{proof}

The following case study discusses some commonly used norms
for the norm in the solution space to obtain some specific
expressions of the condition number $\kappa$. The proof is trivial thus is omitted.

\begin{corollary} \label{colKinfty1}
Using the above notations,
when the infinity norm is chosen as the norm in the solution
space $\cY$, we get
\begin{equation} \label{eqnKinfty1}
\kappa_{\infty} =\left\| \left|\call \bfH\right| \vect (|A|)
+\left|\call \bfJ\right| |b|\right\|_\infty.
\end{equation}
When the infinity norm is chosen as the norm in the solution
space $\R^n$, the corresponding relative mixed condition number is given by
\begin{equation}
\kappa_{\infty}^{\rel} =\frac{\left\| \left|\call \bfH\right| \vect (|A|)
+\left|\call \bfJ\right| |b|\right\|_\infty} {\|\call x\|_\infty}.
\end{equation}
\end{corollary}

In the following, we consider the 2-norm on the solution space and derive an upper bound for the corresponding condition number respect to the 2-norm on the solution space.

\begin{corollary} \label{colK2}
When the 2-norm is used in the solution space, we have
\begin{equation} \label{eqnK2}
\kappa_2 \le
\sqrt{k} \, \kappa_{\infty} .
\end{equation}
\end{corollary}

\begin{proof}
When $\| \cdot \|_\cY= \| \cdot \|_2$, then
$\| \cdot \|_{\cY^\ast} = \| \cdot \|_2$. From Theorem~\ref{thmK},
\[
\kappa_2 = \left\| [ \bfH D_A \ \ \bfJ D_{{b}}]^{\top}
\call^\top\right\|_{2,1} .
\]
It follows from \cite{Higham2002Book} that for any matrix $B$,
$\| B \|_{2,1} = \max_{\| {u} \|_2 = 1} \| B {u} \|_1
= \| B \hat{{u}} \|_1$, where $\hat{{u}} \in \mathbb{R}^k$
is a unit 2-norm vector. Applying $\| \hat{{u}} \|_1 \le
\sqrt{k}\, \| \hat{{u}} \|_2$, we get
\[
\| B \|_{2,1} = \| B \hat{{u}} \|_1 \le
\| B \|_1 \| \hat{{u}} \|_1 \le
\sqrt{k}\, \| B \|_1 .
\]
Substituting the above $B$ with $[ \bfH D_A \ \ \bfJ D_{{b}}]^{\top}
\call^\top$, we have
\[
\kappa_2 \le \sqrt{k} \,
\left\| [ \bfH D_A \ \ \bfJ D_{{b}}]^{\top}
\call^\top\right\|_1 ,
\]
which implies (\ref{eqnK2}). \hfill $\Box$
\end{proof}

By now, we have considered the various mixed condition
numbers, that is, componentwise norm in the data space and
the infinity norm or 2-norm in the solution space. In the
rest of the subsection, we study the case of componentwise
condition number,
that is, componentwise norm in the solution space as well.

\begin{corollary} \label{colKc}
Considering the componentwise norm defined by
\begin{equation}\label{eq:comp norm}
\| {u} \|_c =
\min \{ \omega , \ | u_i | \le \omega \,
|(\call {x})_i|, \ i=1,...,k \} =
\max \{ |u_i| / |(\call {x})_i|, \ i=1,...,k \} ,
\end{equation}
in the solution space, we have the following  expression
for the componentwise condition number
\begin{eqnarray*}
\kappa_c
&=& \left\| |D_{\call {x}}^\dagger ( \left|\call \bfH\right| \vect (|A|)
+\left|\call \bfJ\right| |b|)\right  \|_{\infty}.
\end{eqnarray*}
\end{corollary}

\begin{proof}
The expressions immediately follow from Theorem~\ref{thmK}
and Corollary \ref{colKinfty1}. \hfill $\Box$
\end{proof}

In the following, we will establish the equivalence relationship between $\kappa_{\infty}^{\rel}$,  $\kappa_c$ and  $m(A,b)$, $c(A,b)$ respectively, when $L=I_n$ in \eqref{eq:g dfn}.
\begin{theorem}
	When $L=I_n$, the expressions of $\kappa_{\infty}^{\rel}$ and $\kappa_c$ are equivalent to those of $m(A,b)$ and $c(A,b)$ given by \eqref{eq:mixed}, respectively.
\end{theorem}

\begin{proof}
For $N$ defined in \eqref{eq:mixed}, from \eqref{eq:r}, \eqref{ex:x} and Kronecker product property, it is not difficult to see that
\begin{align*}
	N & =  2\sigma_{n+1}P^{-1}x(v_{n+1}^\top \otimes u_{n+1}^\top)=-2P^{-1}x\,v_{n+1,n+1}^2 \left[x^\top \ \ \   -1\right] \otimes r^\top \\
	&=\frac{1}{1+x^\top x}P^{-1}\left[-2 x( x^\top \otimes r^\top) \ \ \  2xr^\top \right]=\frac{1}{1+x^\top x}P^{-1}\left[x^\top \otimes(-2x r^\top) \ \ \  2xr^\top \right].
\end{align*}
For $M$ given in  \eqref{eq:mixed}, it also can be derived that
\begin{align*}
	M & =  \left[P^{-1}\otimes b^\top-x^\top \otimes(P^{-1}A^\top )-P^{-1}\otimes(Ax)^\top \ \ \  P^{-1}A^\top \right]\\
	&=\left[P^{-1}(I_n \otimes r^\top)- P^{-1} (x^\top \otimes A^\top ) \ \ \ P^{-1}A^\top \right]=P^{-1}\left[(I_n \otimes r^\top)- (x^\top \otimes A^\top )  \ \ \  A^\top \right].
\end{align*}
Combing these two facts and the explicit expressions $\kappa_{\infty}^{\rel}$ and $\kappa_c$  of when $L=I_n$, we can complete the proof of this theorem.
\qed
\end{proof}

The mixed and componentwise condition numbers of a linear function of the TLS solution $x$ can recover the previous results  $m(A,b)$ and $c(A,b)$ given by  \cite{Zhou} when we take $L=I_n$ in the expressions of $\kappa_{\infty}^{\rel}$ and $\kappa_c$. Also, we adopt the dual techniques to derive the condition numbers expressions, which enable us to reduce the computational complexity because  the column number of the matrix expression of $J$ is usually smaller than its row number.

By taking account of the compact form of the inverse of $P$ given in Lemma \ref{lemma:P}, we can give a SVD-based formula of $\kappa_{\infty}^{\rel}$ and $\kappa_c$ in the following corollary.

\begin{corollary}\label{co:un svd}
	With the notations above, we have
	\begin{align*}
		\kappa_{\infty}^{\rel} &=\frac{\left\| \left|\call  Q_1 Q Q_1\left[-x^\top \otimes \left(A^\top+\frac{2xr^\top}{1+x^\top x} \right )+I_n\otimes r^\top \right]\right| \vect (|A|)
+\left|\call Q_1 Q Q_1\left(A^\top+\frac{2xr^\top}{1+x^\top x}\right)\right| |b|\right\|_\infty} {\|\call x\|_\infty},\\
\kappa_c &=\left\| D_{\call x}^\dagger \left|\call  Q_1 Q Q_1\left[-x^\top \otimes \left(A^\top+\frac{2xr^\top}{1+x^\top x} \right )+I_n\otimes r^\top \right]\right| \vect (|A|) \right. \\
& \left. \quad\quad \quad+D_{\call x}^\dagger  \left|\call Q_1 Q Q_1\left(A^\top+\frac{2xr^\top}{1+x^\top x}\right)\right| |b|\right\|_\infty,
	\end{align*}
	where $Q$ and $Q_1$ are defined in Lemma \ref{lemma:P}.
\end{corollary}

Although we have obtained the SVD-based expressions of  $\kappa_{\infty}^{\rel}$ and $\kappa_c$ in Corollary \ref{co:un svd}, they involve the computations of Kronecker product, which may needs extra memory to form them explicitly. In the following, we will give upper bounds for $\kappa_\infty^{\rel}$ and $\kappa_c$ without Kronecker product. The proof of this corollary is based on Kronecker product property and the triangle inequality, and is omitted.

\begin{corollary}\label{co:lsle}
With the notations above, denoting
\begin{align*}
\kappa_\infty^\u&=\frac{\left\| |\call Q_1 Q Q_1|{|A^\top+ \frac{2xr^\top}{1+x^\top x}|   |A||x|}\right\|_\infty+\left\||\call Q_1 Q Q_1|   {|A^\top||r|}\right\|_\infty+\left\|\left|\call Q_1 Q Q_1\left(A^\top+\frac{2xr^\top}{1+x^\top x}\right)\right| |b| \right\|_\infty} {\|\call x\|_\infty},\cr
\kappa_c^\u&=\left\|D_{\call x}^{\dagger }|\call Q_1 Q Q_1 | {|A^\top+ \frac{2xr^\top}{1+x^\top x}| |A||x|}\right\|_\infty+\left\|D_{\call x}^{\dagger } |\call Q_1 Q Q_1|   {|A^\top||r|}\right\|_\infty\\
&\quad+ \left\|D_{\call x}^{\dagger }\left|\call Q_1 Q Q_1\left(A^\top+\frac{2xr^\top}{1+x^\top x}\right)\right| |b|\right\|_\infty,
\end{align*}
we have
 \begin{align*}
   \kappa_\infty^{\rel}\leq \kappa_\infty^\u,\quad  \kappa_c\leq \kappa_c^\u.
  \end{align*}
\end{corollary}

\begin{remark}
  From the numerical results of Example \ref{ex:1} in Section \ref{sec:nume ex}, the upper bounds $\kappa_\infty^\u$ and $\kappa_c^\u $ are asymptotic attainable, thus they are sharp.
\end{remark}

\subsection{Structured condition numbers expressions of TLS via dual techniques}

In this subsection, we will focus on the structured perturbation analysis for the structured TLS problem \cite{BeckSIAM2005,BeckSIAM2007,MastronardiEtal2001,KammNagy1998,LemmerlingVanHuffel2002,LemmerlingVanHuffel2001}. The explicit expressions are deduced and they can recover the previous results on the structured condition numbers given in \cite{LiJia2011}.  Also we will prove that the structured condition numbers are smaller than the corresponding unstructured ones given in the previous subsection from their explicit expressions. We consider $A\in \cal S$ is linear structured, i.e., $A=\sum_{i=1}^{q} a_i S_i $, where $S_1,\ldots, S_q$ form a basis of $\cal S$. Let us denote $ a=[a_1,\ldots, a_q]^\top $. In view of $\sd A= \sum_{i=1}^{q} \sd a_i S_i$, and from Lemma \ref{lemma:gLSQI}, we can prove $\Psi_s$ defined by \eqref{eq:g dfn s} is Fr\'{e}chet differentiable at $(a,\, b)$  and derive its Fr\'{e}chet derivative in the follow lemma.

\begin{lemma}\label{lemma:gLSQIs}
  The function $\Psi_s$ defined by \eqref{eq:g dfn s} is a continuous mapping on $\R^{q} \times \R^{m} $. In addition, $\Psi_s$ is Fr\'{e}chet differentiable at $(A,\, b)$ and its Fr\'{e}chet derivative is given by
\begin{align}\label{eqnJs}
J_{s}&:=\sd {\Psi_s}(a,\,b)\cdot (\sd a,\, \sd b)
=\call P^{-1}V\sd a +\call P^{-1} \left[A^\top+\frac{2xr^\top}{1+x^\top x}\right] \sd b \nonumber \\
&:=J_{1s}(\sd a)+J_{2s}(\sd b),
\end{align}
where $V=[v_1,\ldots, v_q] \in \R^{n \times q}$,  $\sfv_i= S^\rt_{i} r-\left[A^\top+\frac{2xr^\top}{1+x^\top x}\right] S_{i} x$, $\sd a \in \R^{q}$ and $ \sd b \in \R^{m\times 1}$.
\end{lemma}

Lemma \ref{lemmaDualJs} gives the adjoint of operator of $J_s$. Because its proof is similar to Lemma \ref{lemmaDualJ}, it is omitted here.

\begin{lemma}\label{lemmaDualJs}
The adjoint of operator of the Fr{\'e}chet derivative $J_s(\sd a,\, \sd b)$ in \eqref{eqnJs} is given by
\begin{align*}
J^{\ast}_s&:\R^{k}\rightarrow \R^{q} \times \R^{m} \\
 &u\mapsto \left(   V^\top P^{-1}\call^\top u,\quad \left[A^\top+\frac{2xr^\top}{1+x^\top x}\right]^\top P^{-1} \call^\rt u  \right).
\end{align*}
\end{lemma}

The following theorem establishes the expressions of the structured condition number $\kappa_s$ based on the dual techniques. We omit its proof, since it is similar to the proof of Theorem \ref{thmK}.

\begin{theorem} \label{thmKs}
Recalling $\bfJ$ is defined in \eqref{eqnV}, the condition number for the structured TLS
problem can be expressed by
\[
\kappa_s = \max_{\| {u} \|_{\cY^\ast} = 1}
\left\| [ \bfH_s D_a \ \ \bfJ D_{{b}}\ \ ]^{\top}
\call^\top\right\|_{\cY^\ast,1} ,
\]
where $\bfH_s= P^{-1}V$.
\end{theorem}

\begin{corollary} \label{colKinfty1s}
Using the above notations,
when the infinity norm is chosen as the norm in the solution
space $\cY$, we get
\begin{equation} \label{eqnKinfty1}
\kappa_{s,\infty} =\left\| \left|\call \bfH_s\right||a|
+\left|\call \bfJ\right| |b|\right\|_\infty,
\end{equation}
 When the infinity norm is chosen as the norm in the solution
space $\R^n$, the corresponding relative structured mixed condition number is given by
\begin{align*}
\kappa_{s,\infty}^{\rel} &=\frac{\left\| \left|\call \bfH_s\right| |a|
+\left|\call \bfJ\right| |b|\right\|_\infty} {\|\call x\|_\infty}\\
&= \frac{\left\|\sum\limits_{i=1}^{q}|a_i|
\left| \call P^{-1}\left(\left[A^\top +\frac{2xr^\top}{1+x^\top x}\right] S_{i} x-
S^\rt_{i} r\right)\right| +
\left|\call P^{-1}\left[A^\top +\frac{2xr^\top}{1+x^\top x}\right] \right||b|\right\|_\infty}   {\|\call x\|_\infty }.
\end{align*}

\end{corollary}

In the next theorem, we will prove that $\kappa_{s,\infty}^{\rel}$ can recover the expression of  $m_s(A,b)$ given by \eqref{eq:str mixed} when $L=I_n$.

\begin{theorem}
	With the above notations, we have
	$$
	\kappa_{s,\infty}^{\rel} = m_s(A,b),
	$$
	when $L=I_n$ in \eqref{eq:g dfn s}.
\end{theorem}

\begin{proof} From \eqref{svd}, \eqref{ex:x}, \eqref{eq:r} and the fact $v_{n+1,n+1}^2=\frac{1}{1+x^\top x}$, it is easy to verify that
\begin{align*}
  A^\top r &=\frac{\sigma_{n+1}}{ v_{n+1,n+1}}  \left([A,~b]\begin{bmatrix}
    I_n\cr { 0}
  \end{bmatrix}\right)^\top u_{n+1} =\frac{\sigma_{n+1}}{ v_{n+1,n+1}}  V\Sigma U^\top u_{n+1}=  \frac{\sigma_{n+1}^2}{ v_{n+1,n+1}} \begin{bmatrix}
    I_n & { 0}
  \end{bmatrix} v_{n+1}\\&=-   \frac{\|r\|_2^2}{ 1+x^\top x }x .
\end{align*}
Recalling $K$ and $G(x)$ given in \eqref{eq:K} and using Kronecker product property, we can prove that,
\begin{align*}
K &=P^{-1}
\left(2 A^\top \frac{rr^\top }{\|r\|_2^2} G(x)-A^\top G(x)+ \begin{bmatrix}
	I_n \otimes r^\top  & 0
\end{bmatrix} \right )\\
&=P^{-1}
\left(- \frac{2x r^\top }{1+xx^\top } \begin{bmatrix}
	x^\top \otimes  I_m & -I_m
\end{bmatrix}-A^\top \begin{bmatrix}
	x^\top \otimes  I_m & -I_m
\end{bmatrix} + \begin{bmatrix}
	I_n \otimes r^\top  & 0
\end{bmatrix} \right )\\
&=P^{-1}
\left( \begin{bmatrix}
	- x^\top \otimes  \left (    \frac{2x r^\top }{1+xx^\top }\right)& \frac{2x r^\top }{1+xx^\top }
\end{bmatrix}- \begin{bmatrix}
	x^\top \otimes  A^\top & -A^\top \end{bmatrix} + \begin{bmatrix}
	I_n \otimes r^\top  & 0
\end{bmatrix} \right )=\begin{bmatrix} \bfH & \bfJ \end{bmatrix},
\end{align*}
where  $\bfH$ and $\bfJ$ are defined in \eqref{eqnV}. It is not difficult to see that for $V$ defined in Lemma \ref{eqnJs}, $V=\begin{bmatrix}
	I_n \otimes r^\top & x^\top \otimes  \left ( A^\top    \frac{2x r^\top }{1+xx^\top }\right)
\end{bmatrix} \cM^{st}$, where $\cM^{st}$ is defined in \eqref{eq:K}. Then from \eqref{eq:mixed},
\begin{align*}
	m_s(A,b) &= \frac{\left\|\left |K \cM^{st}_{A, b} \right| \begin{bmatrix}
	|a| \\ |b|
\end{bmatrix} \right\|_\infty  }{\|x\|_\infty }=\frac{\left\|\left |\bfH \cM^{st}\right||a|+ |\bfJ| |b| \right\|_\infty  }{\|x\|_\infty }\\
&=\frac{\left\|\left |P^{-1} \begin{bmatrix}
	I_n \otimes r^\top & x^\top \otimes  \left ( A^\top    \frac{2x r^\top }{1+xx^\top }\right)
\end{bmatrix}\cM^{st}\right||a|+ |\bfJ| |b| \right\|_\infty  }{\|x\|_\infty }\\
&=\frac{\left\|\left |P^{-1} V \right||a|+ |\bfJ| |b| \right\|_\infty  }{\|x\|_\infty }= \kappa_{s,\infty}^{\rel},
\end{align*}
whenever $L=I_n$.

\end{proof}\qed

As in the previous section, in the following corollary, we consider the 2-norm on the solution space and derive an upper bound for the corresponding structured condition number respect to the 2-norm on the solution space. The proof is similar to that of  Corollary \ref{colK2}, thus we omit it.

\begin{corollary} \label{colK2s}
When the 2-norm is used in the solution space, we have

\begin{equation*}
\kappa_{s,2} \le
\sqrt{k} \, \kappa_{s,\infty} .
\end{equation*}

\end{corollary}
%

In Corollaries \ref{colKinfty1s} and \ref{colK2s}, we have studied the various mixed condition
numbers, that is, componentwise norm in the data space and
the infinity norm or 2-norm in the solution space. Again as in the previous subsection, we consider the case of componentwise
condition number,
that is, componentwise norm in the solution space as well.
\begin{corollary} \label{colKcs}
Considering the componentwise norm defined by \eqref{eq:comp norm} in the solution space, we have the following two expressions
for the componentwise condition number
\begin{align*}
\kappa_{s,c}
&= \left \| D_{\call {x}}^{\dagger }  \left( \left|\call \bfH_s\right| |a|
+\left|\call \bfJ_s \right| |b| \right)  \right \|_{\infty}\\
&= \left\| D_{\call {x}}^{\dagger }  \left (  \sum\limits_{i=1}^{q}  |a_i|
\left| \call P^{-1} \left( \left[A^\top +\frac{2xr^\top}{1+x^\top x}\right] S_{i} x-
S^\rt_{i} r \right) \right|  \right)+
\left| \call P^{-1}\left[A^\top +\frac{2xr^\top}{1+x^\top x}\right] \right| |b|\right\|_\infty.
\end{align*}
\end{corollary}


In the following theorem, we will prove that the structured mixed and componentwise condition numbers are smaller than the corresponding counterparts from their derived expressions under some assumputions.

\begin{theorem}\label{thm:relation}
Suppose that the basis $\{ S_1,S_2,\ldots, S_q \}$ for $\cal S$ satisfies
$|A|=\sum\limits_{i=1}^q |a_i||S_i|$ for any $A\in \cal S$
in \eqref{eq:A linear}, then
\[
\kappa_{s,\infty}^{\rel} \leq \kappa_{\infty}^{\rel}  \quad \hbox{and} \quad
\kappa_{s,c}\leq \kappa_{c} .
\]	
\end{theorem}

\begin{proof} Using the monotonicity of
the infinity norm, we have

\begin{align*}
& \left\|\sum\limits_{i=1}^{q}|a_i|
\left| \call P^{-1}\left(\left(A^\top +\frac{2xr^\top}{1+x^\top x}\right )S_{i} x-
S^\rt_{i} r\right)\right| +
\left|\call P^{-1}\left (A^\top +\frac{2xr^\top}{1+x^\top x}\right) \right||b|\right\|_\infty\\
&=\left\|\left|\call
\left[\bfH V  \ \ \   \bfJ \right]\right|
\begin{bmatrix}|a|\cr|b|\end{bmatrix}\right\|_\infty
\leq\left\|
\left[ | \call \bfH | ~|V|  \ \ \   |\call \bfJ| \right]
\begin{bmatrix}|a|\cr|b|\end{bmatrix}\right\|_\infty\\
&\leq\left\|\left|\call \bfH\right|
\sum_{i=1}^q |a_i||\vect(S_{i})|+
\left|\call \bfJ \right||b|\right\|_\infty =\left\|\left| \call \bfH \right|\vect(|A|)+
\left|\call \bfJ  \right||b|\right\|_\infty,
\end{align*}
for the last equality we use the assumption $|A|=\sum_{i=1}^q |a_i||S_i|$.
With the above inequality, and the expressions of $\kappa_{s,\infty}^{\rel}$,
$\kappa_{\infty}^{\rel}$, $\kappa_{s,c}$, $\kappa_{c}$,
it is easy to prove the first two inequalities in this theorem.
\qed
\end{proof}

For Toeplitz matrices, the assumption $ |A|=\sum_{i=1}^q |a_i||S_i|$ for $q=m+n-1$ is  satisfied, when
\begin{align*}
	 S_{1}&={\tt toeplitz}(0, e_{n}),\,\ldots, \,
  S_{n}={\tt toeplitz }(0, e_1 ),\\ S_{n+1}&={\tt  toep}(e_2,0 ) \ldots, \,
  S_{m+n-1}={\tt  toeplitz }(e_{m}, 0),
\end{align*}
where \textsc{Matlab}'s notation ${\tt toeplitz}(a,b)$ denotes a Toeplitz matrix with the first column $a$ and first row $b$,   $e_i$ is the $i$th column vector of a conformal dimensional identity matrix and $0$ is the zero vector with a conformal dimension.

By taking account of the compact form of the inverse of $P$ given in Lemma \ref{lemma:P}, we can give  SVD-based formulae of $\kappa_{s,\infty}^{\rel}$ and $\kappa_{s,c}$ in the following corollary.

\begin{corollary}\label{co:un svd}
	With the notations above, we have
	\begin{align*}
		\kappa_{\infty}^{\rel} &=\frac{\left\|  \sum \limits_{i=1}^{q}
|a_i| \left| \call Q_1 Q Q_1 \left(\left[A^\top +\frac{2xr^\top}{1+x^\top x}\right] S_{i} x-
S^\rt_{i} r  \right)  \right|+  \left|\call Q_1 Q Q_1\left(A^\top+\frac{2xr^\top}{1+x^\top x}\right)\right| |b| \right\|_\infty} {\|\call x\|_\infty},\\
\kappa_{s,c}  &=\left\| D_{\call x}^\dagger  \sum \limits_{i=1}^{q}
|a_i| \left| \call Q_1 Q Q_1 \left(\left[A^\top +\frac{2xr^\top}{1+x^\top x}\right] S_{i} x-
S^\rt_{i} r  \right)   \right| \right. \\
&\left. \hspace{1cm}+ D_{\call x}^\dagger    \left| \call Q_1 Q Q_1\left(A^\top+\frac{2xr^\top}{1+x^\top x}\right)\right| |b| \right \|_\infty,
	\end{align*}
	where $Q$ and $Q_1$ are defined in Lemma  \ref{lemma:P}.
\end{corollary}

\section{Numerical examples}\label{sec:nume ex}

In this section we test some numerical examples to validate the previous derived results. All the
computations are carried out using \textsc{Matlab} 8.1 with the machine precision
$\mu=2.2 \times 10^{-16}$.

For a given TLS problem, the TLS solution is computed by \eqref{ex:x}. When the data $A$ and $b$ are generated, for the perturbations, we construct  them as
\begin{equation}\label{eq:pert}
\Delta A=10^{-8} \cdot \Delta A_{1}\odot A,\quad \Delta b=10^{-8}\cdot \Delta b_{1}\odot  b,
\end{equation}
where  each components of $\Delta A_1 \in \R^{m\times n}$ and $\Delta b_1 \in \R^{m}$  are uniformly distributed in the interval $ (-1,1) $,  and $\odot$ denotes the componentwise multiplication of two conformal dimensional matrices. When the perturbations are small enough, we denote the unique solution by $\tilde x$  of the perturbed TLS problem \eqref{eq:pert TLS}.  We use the SVD method  \cite[Algoritm 3.1]{VanHuffelVandewalle1991Book} to compute the solution $x$ and the perturbed solution $\tilde x$ via \eqref{ex:x} separately.

Let $x_{\max} $  and $x_{\min} $ be the maximum and minimum component of $x$ in the absolute vale sense, respectively.  For the $\call$ matrix in our condition numbers, we choose
\[
\call_0=I_n,\quad \call_1=\begin{bmatrix} e_1 & e_2 \end{bmatrix}^\top,\quad
\call_2=e_{\max },\quad
\call_3=e_{\min },
\]
where $\max$ and $\min$ are the indexes corresponding to $x_{\max} $  and $x_{\min} $. Thus, corresponding to the above four matrices, the whole
${x}$, the subvector $[x_1 \ x_2]^{\top }$,
the components $x_{\max} $  and $x_{\min} $ are selected respectively.

We measure the normwise, mixed and componentwise relative errors in
$\call {x}$ defined by
\[
r_2^{\rm rel} =
\frac{\|\call \tilde{{x}} -
\call {x} \|_2}
{\|\call {x} \|_2},\quad
r_\infty^{\rm rel} =
\frac{\|\call \tilde{{x}} -
\call {x} \|_\infty}
{\|\call {x} \|_\infty},\quad
r_c^{\rm rel} =
\frac{\|\call \tilde{{x}} -
\call {x} \|_c}
{\|\call {x} \|_c},
\]
where $\|\cdot\|_c$ is the componentwise norm defined
in \eqref{eq:comp norm}.

\begin{example}\label{ex:1}
We construct the matrix $A$ and  vector $b$ as follows
\begin{align*}
A&=\begin{bmatrix}
  \delta &0&0&0\cr
  0&0&0&0\cr
   0&\delta &0&0\cr
    0&0&0&0\cr
     0&0&0&0\cr
      0&0&0&0\cr
       0&0&1&0\cr
        0&0&0&0\cr
         0&0&0&1
\end{bmatrix} \in \R^{9\times 4},\quad
b=\begin{bmatrix}
  1\cr\vdots \cr 1
\end{bmatrix} \in \R^{9},
\end{align*}
where $\delta$ is a tiny positive parameter. It is easy to see that $A$  is sparse and badly scaled. Thus it is suitable to consider componentwise perturbation analysis for TLS.
\end{example}

From Table \ref{tab:T1}, it is observed that when $\delta $ decrease from $10^{-3}$ to $10^{-9}$,  $\cond^\rel (L,A,b) $ varifies  from the order of $\Oh(10^4)$ to the order of $\Oh(10^{10})$, while $\kappa_\infty^{\rel}$  and  $\kappa_c^{\rel}$ are always $\Oh(1)$. The relative errors  $r_2^\rel$, $r_\infty^{\rel}$ and $r_c^{\rel}$ are tiny, which means that the original TLS problem is well-conditioned. This example indicates it is more suitable to adopt $\kappa_\infty^{\rel}$  and  $\kappa_c^{\rel}$ to measure the conditioning of the TLS problem when the data is sparse or badly-scaled. Moreover, it can be seen that the relative errors can be bounded by the asymptotic first order perturbation bounds based on the proposed condition numbers.  It should be pointed out the upper bounds for $\kappa_\infty^{\rel}$  and  $\kappa_c^{\rel}$ are asymptotic sharp, since they can be attainable from this examples. Also for different choices of $L$, there are differences for the relative errors and condition numbers, which tell us that it is necessary that we should consider the conditioning of the particular interested component by incorporating the matrix $L$ in \eqref{eq:g dfn}.

\begin{landscape}

\begin{table}[!htbp]
\caption{
Comparison of condition numbers with the corresponding relative
errors for Example \ref{ex:1}.
} \label{tab:T1}
\centering
\begin{tabular}{cccccccccc}
\hline
 $\delta$ & $\call$ & $r_2^\rel$  & $\cond^\rel (L,A,b) $ & $r_\infty^{\rel}$ &$\kappa_\infty^{\rel}$ &$\kappa_\infty^\u$& $r_c^{\rel}$&$\kappa_c$ &$\kappa_c^\u$ \\
 \hline
$10^{-3}$&$I_n$ &5.04e-08 &  1.52e+04 &  5.68e-08  &8.43e+00 &  8.43e+00 &  6.85e-09&  8.43e+00 &  8.43e+00 \\
 &$L_1$ &5.04e-08 &  1.52e+04 &  5.68e-08 &  8.43e+00 &  8.43e+00 &  6.85e-09 & 8.43e+00 &  8.43e+00 \\
 & $L_2$ &2.29e-08 &  1.64e+04 &  2.29e-08 & 2.00e+00 &  2.00e+00 &  2.29e-08 & 2.00e+00 &  2.00e+00 \\
 & $L_3$ &4.29e-09 &  1.64e+04 &  4.29e-09&  2.00e+00 &  2.00e+00 &  4.29e-09&  2.00e+00 &  2.00e+00 \\
\hline
$10^{-6}$&$I_n$ & 4.95e-08 &  1.52e+07 &  6.28e-08 &  8.43e+00 &  8.43e+00 &  1.59e-08 &  8.43e+00 &  8.43e+00\\
  &$L_1$ &4.95e-08 &  1.52e+07 &  6.28e-08 &  8.43e+00 &  8.43e+00 &  1.59e-08 &  8.43e+00 &  8.43e+00\\
  &$L_2$ &1.46e-08 &  1.64e+07 &  1.46e-08 &  2.00e+00 &  2.00e+00 &  1.46e-08 &  2.00e+00 &  2.00e+00\\
  &$L_3$ &4.46e-09 &  1.64e+07 &  4.46e-09 &  2.00e+00 &  2.00e+00 &  4.46e-09 &  2.00e+00 &  2.00e+00\\
  \hline
  $10^{-9}$&  $I_n$ & 4.10e-08 &  1.52e+10 &  4.10e-08 &  8.43e+00 &  8.43e+00 &  8.08e-16 &  8.43e+00 &  8.43e+00  \\
  &$L_1$ & 4.10e-08 &  1.52e+10 &  4.10e-08 &  8.43e+00 &  8.43e+00 &  8.08e-16 &  8.43e+00 &  8.43e+00  \\
  &$L_2$ & 3.34e-07 &  1.64e+10 &  3.34e-07 &  2.00e+00 &  2.00e+00 &  3.34e-07 &  2.00e+00 &  2.00e+00  \\
 &$L_3$ & 3.23e-07 &  1.64e+10 &  3.23e-07 &  2.00e+00 &  2.00e+00 &  3.23e-07 &  2.00e+00 &  2.00e+00 \\
 \hline
\end{tabular}
\end{table}

\begin{table}[!htbp]
\caption{
Comparison of condition numbers with the corresponding relative
errors for Example \ref{ExampleBaboulin}.
} \label{tab:T2}
\centering
\begin{tabular}{cccccccccc}
\hline
 $e_p$ & $\call$ & $r_2^\rel$  & $\cond^\rel (L,A,b) $ & $r_\infty^{\rel}$ &$\kappa_\infty^{\rel}$ &$\kappa_\infty^\u$& $r_c^{\rel}$&$\kappa_c$ &$\kappa_c^\u$ \\
 \hline
$10^{0}$&$I_n$& 1.02e-08 &  1.05e+02 &  1.21e-08 &  6.37e+00 &  5.25e+02 &  5.09e-09 &  2.63e+01 &  1.41e+03  \\
&$L_1$ & 7.32e-09 &  6.43e+01 &  5.67e-09 &  1.00e+01 &  1.94e+02 &  3.10e-09 &  1.08e+01 &  1.95e+02  \\
&$L_2$ & 4.82e-09 &  8.24e+01 &  4.82e-09 &  1.03e+01 &  1.94e+02 &  4.82e-09 &  1.03e+01 &  1.94e+02  \\
 &$L_3$ &4.72e-10 &  6.72e+01 &  4.72e-10 &  1.01e+01 &  1.94e+02 &  4.72e-10 &  1.01e+01 &  1.94e+02  \\
\hline
$10^{-4}$&$I_n$ & 1.36e-06 &  6.64e+05 &  1.97e-06 &  4.84e+04 &  4.78e+06 &  6.56e-07 &  4.84e+04 &  4.78e+06  \\
 &$L_1$ & 1.19e-06 &  5.88e+05 &  1.20e-06 &  2.99e+04 &  2.95e+06 &  1.19e-06 &  2.99e+04 &  2.95e+06  \\
  &$L_2$ &1.18e-06 &  5.88e+05 &  1.18e-06 &  2.99e+04 &  2.95e+06 &  1.18e-06 &  2.99e+04 &  2.95e+06  \\
 &$L_3$ & 1.19e-06 &  5.88e+05 &  1.19e-06 &  2.99e+04 &  2.95e+06 &  1.19e-06 &  2.99e+04 &  2.95e+06  \\
  \hline
  $10^{-8}$&  $I_n$ & 4.88e-02 &  6.67e+09 &  8.57e-02 &  5.15e+08 &  1.49e+10 &  6.06e-03 &  2.79e+09 &  8.07e+10  \\
&$L_1$ & 1.10e-02 &  1.51e+09 &  1.10e-02 &  6.63e+07 &  1.92e+09 &  1.10e-02 &  6.63e+07 &  1.92e+09  \\
 &$L_2$ &1.10e-02 &  1.51e+09 &  1.10e-02 &  6.63e+07 &  1.92e+09 &  1.10e-02 &  6.63e+07 &  1.92e+09  \\
 &$L_3$ &1.10e-02 &  1.51e+09 &  1.10e-02 &  6.63e+07 &  1.92e+09 &  1.10e-02 &  6.63e+07 &  1.92e+09\\
 \hline
\end{tabular}
\end{table}

\end{landscape}

\begin{example} \label{ExampleBaboulin}
 We adopt the example from \cite{Baboulin2011SIMAX}, i.e.,
$$
\left[A, ~b\right]=Y\begin{bmatrix}D\cr 0\end{bmatrix} Z^\top \in \R^{m\times(n+1)}, \quad Y=I_m-2yy^\top,\quad Z=I_{n+1}-2zz^\top
$$
where $y\in\R^m$ and $z\in \R^{n+1}$ are random unit vectors and $D={\diag}(n,n-1,\ldots,1,1-e_p)$ for  a given parameter $e_p$.
We set $m=100,\, n=20$ as \cite{Baboulin2011SIMAX}. In this example, the matrix $A$ and $b$ are usually not spares and badly-scaled from their forms. So it cannot be argued that there should be big differences between the normwise condition numbers and mixed/componentwise condition numbers.
\end{example}

From the observation of Table \ref{tab:T2}, we can conclude that when $e_p$ becomes smaller, the corresponding TLS problem tends to be more ill-conditioned. For example, when $e_p=10^{-8}$, the small perturbations on $A$ and $b$ cause big relative errors for the TLS solution $x$. Because the generated data $A$ and $b$ is usually not sparse or badly-scaled, there are no big differences between $\cond^\rel(L,A,b)$ and $\kappa_\infty^{\rel}$ ($\kappa_c$). However, the values of  $\kappa_\infty^{\rel}$  and $\kappa_c$ are smaller to the corresponding parts of $\cond^\rel(L,A,b)$  for different $L$ and $e_p$. And the asymptotic first order perturbation bounds based on $\kappa_\infty^{\rel}$  and $\kappa_c$  are sharper than the ones given by $\cond^\rel(L,A,b)$.  Also, the upper bounds $\kappa_\infty^\u$ and $\kappa_c^\u$ are effective, since they are at most one hundredfold of $\kappa_\infty^{\rel}$  and $\kappa_c$, respectively.



\begin{example}\label{ex:str}
This example is taken from \cite{KammNagy1998}, which is from the application in signal restoration. Let $\alpha = 1.25$
and $\omega = 8$. The convolution matrix $\bar{A}$ is an
$m \times (m-2\omega)$ Toeplitz matrix with entries in the first column given by
$$
a_{i1} = \frac{1}{\sqrt{2 \pi \alpha^2}} \exp \left[ \frac{-(\omega - i + 1)^2}{2\alpha^2} \right], ~~ i = 1,2,\ldots,2\omega+1,
$$
and $a_{i1} = 0$ otherwise. The entries in the row are all zeros except $a_{11}.$
The target Toeplitz matrix $A$ and right-hand side vector $b$ are then constructed as
$$A = \bar{A} + E ~\mbox{and}~  b = \bar{b} + e$$
where $\bar{b}$ is the vector of all ones and $E$ is a random Toeplitz matrix with the same structure as $\bar{A}.$ The entries in $E$
and $e$ are generated from the standard normal distribution and scaled such that
$$\frac{\|e\|_2}{\|\bar{b}\|_2} = \frac{\|E\|_2}{\|\bar{A}\|_2} = \gamma.$$
In our test, we take $\gamma = 0.001$ and $m=200$.
\end{example}


\begin{table}[H]
\caption{
Comparison of condition numbers with the corresponding relative
errors for Example \ref{ex:str}.
} \label{tab:T3}
\centering
\begin{tabular}{cccccccccc}
\hline
 $\call$  & $r_\infty^{\rel}$ &$\kappa_\infty^{\rel}$ &$\kappa_{s,\infty}^{\rel}$& $r_c^{\rel}$&$\kappa_c$ &$\kappa_{s,c}$ \\
 \hline
$I_n$& 5.14e-06 &  3.30e+04 &  2.50e+02 &  1.94e-06 &  4.26e+06 &  4.37e+03  \\
 $L_1$& 6.55e-06 &  4.87e+04 &  1.69e+02 &  6.55e-06 &  4.87e+04 &  1.69e+02  \\
 $L_2$& 6.21e-06 &  4.67e+04 &  1.75e+02 &  6.21e-06 &  4.67e+04 &  1.75e+02  \\
 $L_3$& 5.77e-06 &  4.42e+04 &  1.80e+02 &  5.77e-06 &  4.42e+04 &  1.80e+02  \\
\hline
\end{tabular}
\end{table}


In Table \ref{ex:str}, the structured mixed and componentwise condition numbers are always smaller than the corresponding unstructured mixed and componentwise counterparts. The maximum and minimum ratios between the unstructured  and structured ones are of $\Oh(10^3)$ and $\Oh(10^2)$, respectively. So it is suitable to consider the structured perturbation analysis  and measure the structured conditioning instead of the unstructured conditioning for the structured TLS problem.

\section{Concluding Remarks}\label{sec:con}

In this paper we focused on  the unstructured and structured componentwise perturbation analysis for the TLS problem.  Condition number expressions for the linear function of the TLS solution were derived through the dual techniques under componentwise perturbations for the input data. Moreover we studied the relationship between the new derived ones and the previous results. Sharp upper bounds for the unstructured condition numbers were given. We had proved that the structured condition numbers are smaller than the unstructured ones from the derived explicit expressions. Numerical examples validated our theoretical results.


%

\begin{thebibliography}{37}
\expandafter\ifx\csname natexlab\endcsname\relax\def\natexlab#1{#1}\fi
\providecommand{\url}[1]{\texttt{#1}}
\providecommand{\href}[2]{#2}
\providecommand{\path}[1]{#1}
\providecommand{\DOIprefix}{doi:}
\providecommand{\ArXivprefix}{arXiv:}
\providecommand{\URLprefix}{URL: }
\providecommand{\Pubmedprefix}{pmid:}
\providecommand{\doi}[1]{\href{http://dx.doi.org/#1}{\path{#1}}}
\providecommand{\Pubmed}[1]{\href{pmid:#1}{\path{#1}}}
\providecommand{\bibinfo}[2]{#2}
\ifx\xfnm\relax \def\xfnm[#1]{\unskip,\space#1}\fi
\bibitem[{Golub and Van~Loan(1980)}]{GolubTLS1980}
\bibinfo{author}{G.~H. Golub}, \bibinfo{author}{C.~F. Van~Loan},
\newblock \bibinfo{title}{An analysis of the total least squares problem},
\newblock \bibinfo{journal}{SIAM J. Numer. Anal.} \bibinfo{volume}{17}
  (\bibinfo{year}{1980}) \bibinfo{pages}{883--893}.
\bibitem[{Golub and Van~Loan(2013)}]{GolubVanLoan2013Book}
\bibinfo{author}{G.~H. Golub}, \bibinfo{author}{C.~F. Van~Loan},
  \bibinfo{title}{Matrix computations}, Johns Hopkins Studies in the
  Mathematical Sciences, \bibinfo{edition}{fourth} ed.,
  \bibinfo{publisher}{Johns Hopkins University Press, Baltimore, MD},
  \bibinfo{year}{2013}.
\bibitem[{Van~Huffel and Vandewalle(1991)}]{VanHuffelVandewalle1991Book}
\bibinfo{author}{S.~Van~Huffel}, \bibinfo{author}{J.~Vandewalle},
  \bibinfo{title}{The total least squares problem}, volume~\bibinfo{volume}{9}
  of \textit{\bibinfo{series}{Frontiers in Applied Mathematics}},
  \bibinfo{publisher}{Society for Industrial and Applied Mathematics (SIAM),
  Philadelphia, PA}, \bibinfo{year}{1991}. \URLprefix
  \url{http://dx.doi.org/10.1137/1.9781611971002}.
  \DOIprefix\doi{10.1137/1.9781611971002}, \bibinfo{note}{computational aspects
  and analysis, With a foreword by Gene H. Golub}.
\bibitem[{Markovsky and Huffel(2007)}]{MarkovskyVanHuffel2007}
\bibinfo{author}{I.~Markovsky}, \bibinfo{author}{S.~V. Huffel},
\newblock \bibinfo{title}{Overview of total least-squares methods},
\newblock \bibinfo{journal}{Signal Processing} \bibinfo{volume}{87}
  (\bibinfo{year}{2007}) \bibinfo{pages}{2283 -- 2302}. \bibinfo{note}{Special
  Section: Total Least Squares and Errors-in-Variables Modeling}.
\bibitem[{Bj{\"o}rck et~al.(2000)Bj{\"o}rck, Heggernes, and
  Matstoms}]{Bjorck2000SIMAX}
\bibinfo{author}{{\AA}.~Bj{\"o}rck}, \bibinfo{author}{P.~Heggernes},
  \bibinfo{author}{P.~Matstoms},
\newblock \bibinfo{title}{Methods for large scale total least squares
  problems},
\newblock \bibinfo{journal}{SIAM J. Matrix Anal. Appl.} \bibinfo{volume}{22}
  (\bibinfo{year}{2000}) \bibinfo{pages}{413--429 (electronic)}.
\bibitem[{Beck and Ben-Tal(2005)}]{BeckSIAM2005}
\bibinfo{author}{A.~Beck}, \bibinfo{author}{A.~Ben-Tal},
\newblock \bibinfo{title}{A global solution for the structured total least
  squares problem with block circulant matrices},
\newblock \bibinfo{journal}{SIAM J. Matrix Anal. Appl.} \bibinfo{volume}{27}
  (\bibinfo{year}{2005}) \bibinfo{pages}{238--255}.
\bibitem[{Beck et~al.(2007)Beck, Eldar, and Ben-Tal}]{BeckSIAM2007}
\bibinfo{author}{A.~Beck}, \bibinfo{author}{Y.~C. Eldar},
  \bibinfo{author}{A.~Ben-Tal},
\newblock \bibinfo{title}{Mean-squared error estimation for linear systems with
  block circulant uncertainty},
\newblock \bibinfo{journal}{SIAM J. Matrix Anal. Appl.} \bibinfo{volume}{29}
  (\bibinfo{year}{2007}) \bibinfo{pages}{712--730}.
\bibitem[{Mastronardi et~al.(2001)Mastronardi, Lemmerling, and
  Van~Huffel}]{MastronardiEtal2001}
\bibinfo{author}{N.~Mastronardi}, \bibinfo{author}{P.~Lemmerling},
  \bibinfo{author}{S.~Van~Huffel},
\newblock \bibinfo{title}{The structured total least squares problem},
\newblock in: \bibinfo{booktitle}{Structured matrices in mathematics, computer
  science, and engineering, {I} ({B}oulder, {CO}, 1999)}, volume
  \bibinfo{volume}{280} of \textit{\bibinfo{series}{Contemp. Math.}},
  \bibinfo{publisher}{Amer. Math. Soc., Providence, RI}, \bibinfo{year}{2001},
  pp. \bibinfo{pages}{157--175}. \URLprefix
  \url{http://dx.doi.org/10.1090/conm/280/04627}.
  \DOIprefix\doi{10.1090/conm/280/04627}.
\bibitem[{Kamm and Nagy(1998)}]{KammNagy1998}
\bibinfo{author}{J.~Kamm}, \bibinfo{author}{J.~G. Nagy},
\newblock \bibinfo{title}{A total least squares method for {T}oeplitz systems
  of equations},
\newblock \bibinfo{journal}{BIT} \bibinfo{volume}{38} (\bibinfo{year}{1998})
  \bibinfo{pages}{560--582}.
\bibitem[{Lemmerling et~al.(2002)Lemmerling, Van~Huffel, and
  De~Moor}]{LemmerlingVanHuffel2002}
\bibinfo{author}{P.~Lemmerling}, \bibinfo{author}{S.~Van~Huffel},
  \bibinfo{author}{B.~De~Moor},
\newblock \bibinfo{title}{The structured total least-squares approach for
  non-linearly structured matrices},
\newblock \bibinfo{journal}{Numer. Linear Algebra Appl.} \bibinfo{volume}{9}
  (\bibinfo{year}{2002}) \bibinfo{pages}{321--332}.
\bibitem[{Lemmerling and Van~Huffel(2001)}]{LemmerlingVanHuffel2001}
\bibinfo{author}{P.~Lemmerling}, \bibinfo{author}{S.~Van~Huffel},
\newblock \bibinfo{title}{Analysis of the structured total least squares
  problem for {H}ankel/{T}oeplitz matrices},
\newblock \bibinfo{journal}{Numer. Algorithms} \bibinfo{volume}{27}
  (\bibinfo{year}{2001}) \bibinfo{pages}{89--114}.
\bibitem[{Higham(2002)}]{Higham2002Book}
\bibinfo{author}{N.~J. Higham}, \bibinfo{title}{Accuracy and stability of
  numerical algorithms}, \bibinfo{edition}{second} ed.,
  \bibinfo{publisher}{SIAM, Philadelphia, PA}, \bibinfo{year}{2002}. \URLprefix
  \url{http://dx.doi.org/10.1137/1.9780898718027}.
  \DOIprefix\doi{10.1137/1.9780898718027}.
\bibitem[{Fierro and Bunch(1996)}]{FierroBunch}
\bibinfo{author}{R.~D. Fierro}, \bibinfo{author}{J.~R. Bunch},
\newblock \bibinfo{title}{Perturbation theory for orthogonal projection methods
  with applications to least squares and total least squares},
\newblock \bibinfo{journal}{Linear Algebra Appl.} \bibinfo{volume}{234}
  (\bibinfo{year}{1996}) \bibinfo{pages}{71--96}.
\bibitem[{Wei(1992)}]{Weib}
\bibinfo{author}{M.~S. Wei},
\newblock \bibinfo{title}{The analysis for the total least squares problem with
  more than one solution},
\newblock \bibinfo{journal}{SIAM J. Matrix Anal. Appl.} \bibinfo{volume}{13}
  (\bibinfo{year}{1992}) \bibinfo{pages}{746--763}.
\bibitem[{Rice(1966)}]{Rice}
\bibinfo{author}{J.~R. Rice},
\newblock \bibinfo{title}{A theory of condition},
\newblock \bibinfo{journal}{SIAM J. Numer. Anal.} \bibinfo{volume}{3}
  (\bibinfo{year}{1966}) \bibinfo{pages}{287--310}.
\bibitem[{Higham(1994)}]{Higham1993CompSurvey}
\bibinfo{author}{N.~J. Higham},
\newblock \bibinfo{title}{A survey of componentwise perturbation theory in
  numerical linear algebra},
\newblock in: \bibinfo{booktitle}{Mathematics of {C}omputation 1943--1993: a
  half-century of computational mathematics ({V}ancouver, {BC}, 1993)},
  volume~\bibinfo{volume}{48} of \textit{\bibinfo{series}{Proc. Sympos. Appl.
  Math.}}, \bibinfo{publisher}{Amer. Math. Soc., Providence, RI},
  \bibinfo{year}{1994}, pp. \bibinfo{pages}{49--77}. \URLprefix
  \url{http://dx.doi.org/10.1090/psapm/048/1314843}.
  \DOIprefix\doi{10.1090/psapm/048/1314843}.
\bibitem[{Anderson et~al.(1999)Anderson, Bai, Bischof, Blackford, Demmel,
  Dongarra, Du~Croz, Hammarling, Greenbaum, McKenney, and
  Sorensen}]{Anderson1999Lapack}
\bibinfo{author}{E.~Anderson}, \bibinfo{author}{Z.~Bai},
  \bibinfo{author}{C.~Bischof}, \bibinfo{author}{L.~S. Blackford},
  \bibinfo{author}{J.~Demmel}, \bibinfo{author}{J.~J. Dongarra},
  \bibinfo{author}{J.~Du~Croz}, \bibinfo{author}{S.~Hammarling},
  \bibinfo{author}{A.~Greenbaum}, \bibinfo{author}{A.~McKenney},
  \bibinfo{author}{D.~Sorensen}, \bibinfo{title}{LAPACK Users' Guide (Third
  Ed.)}, \bibinfo{publisher}{Society for Industrial and Applied Mathematics},
  \bibinfo{address}{Philadelphia, PA, USA}, \bibinfo{year}{1999}.
\bibitem[{Cucker et~al.(2007)Cucker, Diao, and Wei}]{CuckerDiaoWei2007MC}
\bibinfo{author}{F.~Cucker}, \bibinfo{author}{H.~Diao},
  \bibinfo{author}{Y.~Wei},
\newblock \bibinfo{title}{On mixed and componentwise condition numbers for
  {M}oore-{P}enrose inverse and linear least squares problems},
\newblock \bibinfo{journal}{Math. Comp.} \bibinfo{volume}{76}
  (\bibinfo{year}{2007}) \bibinfo{pages}{947--963}.
\bibitem[{Gohberg and Koltracht(1993)}]{GohbergKoltracht1993SIMAX}
\bibinfo{author}{I.~Gohberg}, \bibinfo{author}{I.~Koltracht},
\newblock \bibinfo{title}{Mixed, componentwise, and structured condition
  numbers},
\newblock \bibinfo{journal}{SIAM J. Matrix Anal. Appl.} \bibinfo{volume}{14}
  (\bibinfo{year}{1993}) \bibinfo{pages}{688--704}.
\bibitem[{Rohn(1989)}]{Rohn89}
\bibinfo{author}{J.~Rohn},
\newblock \bibinfo{title}{New condition numbers for matrices and linear
  systems},
\newblock \bibinfo{journal}{Computing} \bibinfo{volume}{41}
  (\bibinfo{year}{1989}) \bibinfo{pages}{167--169}.
\bibitem[{Skeel(1979)}]{Skeel79}
\bibinfo{author}{R.~D. Skeel},
\newblock \bibinfo{title}{Scaling for numerical stability in {Gaussian}
  elimination},
\newblock \bibinfo{journal}{J. ACM} \bibinfo{volume}{26} (\bibinfo{year}{1979})
  \bibinfo{pages}{494--526}.
\bibitem[{Baboulin and Gratton(2011)}]{Baboulin2011SIMAX}
\bibinfo{author}{M.~Baboulin}, \bibinfo{author}{S.~Gratton},
\newblock \bibinfo{title}{A contribution to the conditioning of the total
  least-squares problem},
\newblock \bibinfo{journal}{SIAM J. Matrix Anal. Appl.} \bibinfo{volume}{32}
  (\bibinfo{year}{2011}) \bibinfo{pages}{685--699}.
\bibitem[{Jia and Li(2013)}]{JiaLi2013}
\bibinfo{author}{Z.~Jia}, \bibinfo{author}{B.~Li},
\newblock \bibinfo{title}{On the condition number of the total least squares
  problem},
\newblock \bibinfo{journal}{Numer. Math.} \bibinfo{volume}{125}
  (\bibinfo{year}{2013}) \bibinfo{pages}{61--87}.
\bibitem[{Zhou et~al.(2009)Zhou, Lin, Wei, and Qiao}]{Zhou}
\bibinfo{author}{L.~Zhou}, \bibinfo{author}{L.~Lin}, \bibinfo{author}{Y.~Wei},
  \bibinfo{author}{S.~Qiao},
\newblock \bibinfo{title}{Perturbation analysis and condition numbers of scaled
  total least squares problems},
\newblock \bibinfo{journal}{Numer. Algorithms} \bibinfo{volume}{51}
  (\bibinfo{year}{2009}) \bibinfo{pages}{381--399}.
\bibitem[{Diao et~al.(2016)Diao, Wei, and Xie}]{DiaoWeiXie}
\bibinfo{author}{H.-A. Diao}, \bibinfo{author}{Y.~Wei},
  \bibinfo{author}{P.~Xie},
\newblock \bibinfo{title}{Small sample statistical condition estimation for the
  total least squares problem},
\newblock \bibinfo{journal}{Numer. Algor.}  (\bibinfo{year}{2016}).
  \bibinfo{note}{DOI:10.1007/s11075-016-0185-9}.
\bibitem[{Li and Jia(2011)}]{LiJia2011}
\bibinfo{author}{B.~Li}, \bibinfo{author}{Z.~Jia},
\newblock \bibinfo{title}{Some results on condition numbers of the scaled total
  least squares problem},
\newblock \bibinfo{journal}{Linear Algebra Appl.} \bibinfo{volume}{435}
  (\bibinfo{year}{2011}) \bibinfo{pages}{674--686}.
\bibitem[{Higham and Higham(1999)}]{Higham1998structured}
\bibinfo{author}{D.~J. Higham}, \bibinfo{author}{N.~J. Higham},
\newblock \bibinfo{title}{Structured backward error and condition of
  generalized eigenvalue problems},
\newblock \bibinfo{journal}{SIAM J. Matrix Anal. Appl.} \bibinfo{volume}{20}
  (\bibinfo{year}{1999}) \bibinfo{pages}{493--512 (electronic)}.
\bibitem[{Rump(2003{\natexlab{a}})}]{Rump03a}
\bibinfo{author}{S.~M. Rump},
\newblock \bibinfo{title}{Structured perturbations. {I}. {N}ormwise distances},
\newblock \bibinfo{journal}{SIAM J. Matrix Anal. Appl.} \bibinfo{volume}{25}
  (\bibinfo{year}{2003}{\natexlab{a}}) \bibinfo{pages}{1--30 (electronic)}.
\bibitem[{Rump(2003{\natexlab{b}})}]{Rump03b}
\bibinfo{author}{S.~M. Rump},
\newblock \bibinfo{title}{Structured perturbations. {II}. {C}omponentwise
  distances},
\newblock \bibinfo{journal}{SIAM J. Matrix Anal. Appl.} \bibinfo{volume}{25}
  (\bibinfo{year}{2003}{\natexlab{b}}) \bibinfo{pages}{31--56 (electronic)}.
\bibitem[{Cucker and Diao(2007)}]{CuckerDiao2007Calcolo}
\bibinfo{author}{F.~Cucker}, \bibinfo{author}{H.~Diao},
\newblock \bibinfo{title}{Mixed and componentwise condition numbers for
  rectangular structured matrices},
\newblock \bibinfo{journal}{Calcolo} \bibinfo{volume}{44}
  (\bibinfo{year}{2007}) \bibinfo{pages}{89--115}.
\bibitem[{Diao et~al.(2016)Diao, Wei, and Qiao}]{DiaoWeiQiao2016}
\bibinfo{author}{H.-A. Diao}, \bibinfo{author}{Y.~Wei},
  \bibinfo{author}{S.~Qiao},
\newblock \bibinfo{title}{Structured condition numbers of structured {T}ikhonov
  regularization problem and their estimations},
\newblock \bibinfo{journal}{J. Comput. Appl. Math.} \bibinfo{volume}{308}
  (\bibinfo{year}{2016}) \bibinfo{pages}{276--300}.
\bibitem[{Xu et~al.(2006)Xu, Wei, and Qiao}]{XuWeiQiao}
\bibinfo{author}{W.~Xu}, \bibinfo{author}{Y.~Wei}, \bibinfo{author}{S.~Qiao},
\newblock \bibinfo{title}{Condition numbers for structured least squares
  problems},
\newblock \bibinfo{journal}{BIT} \bibinfo{volume}{46} (\bibinfo{year}{2006})
  \bibinfo{pages}{203--225}.
\bibitem[{Chandrasekaran and Ipsen(1995)}]{ChandrasekaranIpsen1993}
\bibinfo{author}{S.~Chandrasekaran}, \bibinfo{author}{I.~C.~F. Ipsen},
\newblock \bibinfo{title}{On the sensitivity of solution components in linear
  systems of equations},
\newblock \bibinfo{journal}{SIAM J. Matrix Anal. Appl.} \bibinfo{volume}{16}
  (\bibinfo{year}{1995}) \bibinfo{pages}{93--112}.
\bibitem[{Arioli et~al.(2007)Arioli, Baboulin, and
  Gratton}]{ArioliBaboulinGratton2007}
\bibinfo{author}{M.~Arioli}, \bibinfo{author}{M.~Baboulin},
  \bibinfo{author}{S.~Gratton},
\newblock \bibinfo{title}{A partial condition number for linear least squares
  problems},
\newblock \bibinfo{journal}{SIAM J. Matrix Anal. Appl.} \bibinfo{volume}{29}
  (\bibinfo{year}{2007}) \bibinfo{pages}{413--433}.
\bibitem[{Baboulin and Gratton(2009)}]{BaboulinGratton2009BIT}
\bibinfo{author}{M.~Baboulin}, \bibinfo{author}{S.~Gratton},
\newblock \bibinfo{title}{Using dual techniques to derive componentwise and
  mixed condition numbers for a linear function of a linear least squares
  solution},
\newblock \bibinfo{journal}{BIT} \bibinfo{volume}{49} (\bibinfo{year}{2009})
  \bibinfo{pages}{3--19}.
\bibitem[{Baboulin et~al.(2009)Baboulin, Dongarra, Gratton, and
  Langou}]{BaboulinDongarraGratton2009}
\bibinfo{author}{M.~Baboulin}, \bibinfo{author}{J.~Dongarra},
  \bibinfo{author}{S.~Gratton}, \bibinfo{author}{J.~Langou},
\newblock \bibinfo{title}{Computing the conditioning of the components of a
  linear least-squares solution},
\newblock \bibinfo{journal}{Numer. Linear Algebra Appl.} \bibinfo{volume}{16}
  (\bibinfo{year}{2009}) \bibinfo{pages}{517--533}.
\bibitem[{Graham(1981)}]{Graham1981book}
\bibinfo{author}{A.~Graham}, \bibinfo{title}{Kronecker products and matrix
  calculus: with applications}, \bibinfo{publisher}{Ellis Horwood Ltd.,
  Chichester; Halsted Press [John Wiley \& Sons, Inc.], New York},
  \bibinfo{year}{1981}. \bibinfo{note}{Ellis Horwood Series in Mathematics and
  its Applications}.

\end{thebibliography}

\end{document}